\documentclass{amsart}

 \usepackage[foot]{amsaddr}
\let\oldtocsection=\tocsection
 
\let\oldtocsubsection=\tocsubsection 
 
\let\oldtocsubsubsection=\tocsubsubsection
 
\renewcommand{\tocsection}[2]{\vspace{0.5em}\hspace{0em}\oldtocsection{#1}{#2}}
\renewcommand{\tocsubsection}[2]{\vspace{0.5em}\hspace{1em}\oldtocsubsection{#1}{#2}}
\renewcommand{\tocsubsubsection}[2]{\vspace{0.5em}\hspace{2em}\oldtocsubsubsection{#1}{#2}}
\usepackage{graphicx,cancel,xcolor,hyperref,comment,graphicx,geometry}
 
\usepackage{tikz}
\usepackage{graphicx,url,etoolbox}
\usepackage{lipsum}
\usepackage{dsfont}

\usepackage{fancyhdr}
\usepackage{lastpage}

\newtheorem{theoreme}{Theorem}[section]

\newtheorem{rem}[theoreme]{Remark}

\theoremstyle{definition}

\numberwithin{equation}{section}

 \renewenvironment{proof}{{\bfseries \noindent Proof.}}{\demo}
\newcommand\xqed[1]{%
  \leavevmode\unskip\penalty9999 \hbox{}\nobreak\hfill
  \quad\hbox{#1}}
\newcommand\demo{\xqed{$\square$}}

\hypersetup{bookmarks, colorlinks, urlcolor=blue, citecolor=red, linkcolor=blue, hyperfigures, pagebackref,
    pdfcreator=LaTeX, breaklinks=true, pdfpagelayout=SinglePage, bookmarksopen=true,bookmarksopenlevel=2}

\usepackage{comment}
\def\u2{\u^2}
\def\u3{\u^3}
\def\u4{\u^4}
\def\u5{\u^5}
\def\y1{\y^1}
\def\y2{\z^1}
\def\y3{\y^3}
\def\y4{\z^3}
\def\y5{\y^5}

\def\R{\mathbb R}

\def\HH{\mathcal H}

\def\AA{\mathcal A}

\def\la {{\lambda}}

\newcommand {\nc}   {\newcommand}
\nc {\be}   {\begin{equation}} \nc {\ee}   {\end{equation}} \nc
{\beq}  {\begin{eqnarray}} \nc {\eeq}  {\end{eqnarray}} \nc {\beqs}
{\begin{eqnarray*}} \nc {\eeqs} {\end{eqnarray*}}
\def\edc{\end{document}}

\providecommand{\abs}[1]{\lvert#1\rvert}

\usepackage{tikz}
\usetikzlibrary{decorations.pathmorphing,patterns,scopes,intersections,calc}
\usepackage{caption}

\newcounter{dummy} 
\numberwithin{dummy}{section}
\newtheorem{Theorem}[dummy]{Theorem}

\newtheorem{Corollary}[dummy]{Corollary}

\newtheorem{defi}[dummy]{Definition}

\newtheorem{Lemma}[dummy]{Lemma}

\numberwithin{equation}{section}

\begin{document}
\title[\fontsize{7}{9}\selectfont  ]{Stabilization of the generalized Rao-Nakra beam by partial viscous damping}
\author{Mohammad Akil$^{1}$ and Zhuangyi Liu$^{2}$}
\address{$^1$ Universit\'e Savoie Mont Blanc - Chamb\'ery - France, Laboratoire LAMA}
\address{$2$ University of Minnesota - Duluth, Department of Mathematics and Statistics.}
\email{mohammad.akil@univ-smb.fr ,  zliu@d.umn.edu}

\maketitle
\begin{abstract}
{In this paper, we consider the stabilization of the generalized Rao-Nakra beam equation, which consists of four wave equations for the longitudinal displacements and the shear angle of the top and bottom layers and one Euler-Bernoulli beam equation for the transversal displacement. Dissipative mechanism are provided through viscous damping for two displacements. The location of the viscous damping are divided into two groups, characterized by whether both of the top and bottom layers are directly damped or otherwise. Each group consists of three cases. We obtain the necessary and sufficient conditions for the cases in group two to be strongly stable. Furthermore, polynomial stability of certain orders are proved. The cases in group one are left for future study. }

\end{abstract}
\pagenumbering{roman}
\maketitle
\tableofcontents
\pagenumbering{arabic}
\setcounter{page}{1}

\section{Introduction}
\noindent Several three layer laminated beam and plate models were proposed in the late 1960's and early 1970's (\cite{MEAD1969163}, \cite{SADASIVARAO1974309} and \cite{10.1115/1.3422825}).  Later, {the following generalized Nakra-Rao beam model was developed in \cite{LIU1999149}, where the shear effect of the bottom and top layer are taken into account. }
\begin{equation}\label{1LB-1}
\left\{\begin{array}{lll}
\rho_1h_1u^1_{tt}-E_1h_1u^1_{xx}-\tau=0&\text{in}&(0,L)\times \R^{+},\\[0.1in]
\rho_1I_1y^1_{tt}-E_1I_1y^1_{xx}-\frac{h_1}{2}\tau+G_1h_1(\omega_x+y^1)=0&\text{in}&(0,L)\times \R^{+},\\[0.1in]
\rho h\omega_{tt}+EI\omega_{xxxx}-G_1h_1(\omega_x+y^1)_x-G_3h_3(\omega_x+y^3)_x-h_2\tau_x=0&\text{in}&(0,L)\times \R^{+},\\[0.1in]
\rho_3h_3u^3_{tt}-E_3h_3u^3_{xx}+\tau=0&\text{in}&(0,L)\times \R^{+},\\[0.1in]
\rho_3I_3y^3_{tt}-E_3I_3y^3_{xx}-\frac{h_3}{2}\tau+G_3h_3(\omega_x+y^3)=0&\text{in}&(0,L)\times \R^{+},
\end{array}
\right.
\end{equation}
where { $u^1,y^1,u^3,y^3$ are the longitudinal displacement and shear angle of the top and bottom layers};\ $\omega$ is the transverse displacement of the beam, $\tau$ is the shear stress in the core layer, where 
$$
\tau=-u^1-\frac{h_1}{2}y^1+h_2\omega_x+u^3-\frac{h_3}{2}y^3.
$$
The physical parameters $\rho_1, E_i, G_i,I_i>0$ are the  density, Young's modulus, shear modulus, and moments of inertia of the { $ith$ layer for $i=1,2,3$,  respectively}. { The beam is composed of a top and a bottom face-plate of respective thicknesses, $h^1$ and $h^3$ and a core of thickness $h^2$}. In addition, $\rho h=\rho_1h_1+\rho_2h_2+\rho_3h_3$, $EI=E_1I_1+E_3I_3$,{ and the shear modulus $G_2=\frac{E_2}{2(1+\nu)}$ with $-1<\nu<\frac{1}{2}$ being the Poisson ratio} (see Figure 1).
\begin{figure}[h]
\begin{center}
\includegraphics[height=15cm,width=15cm]{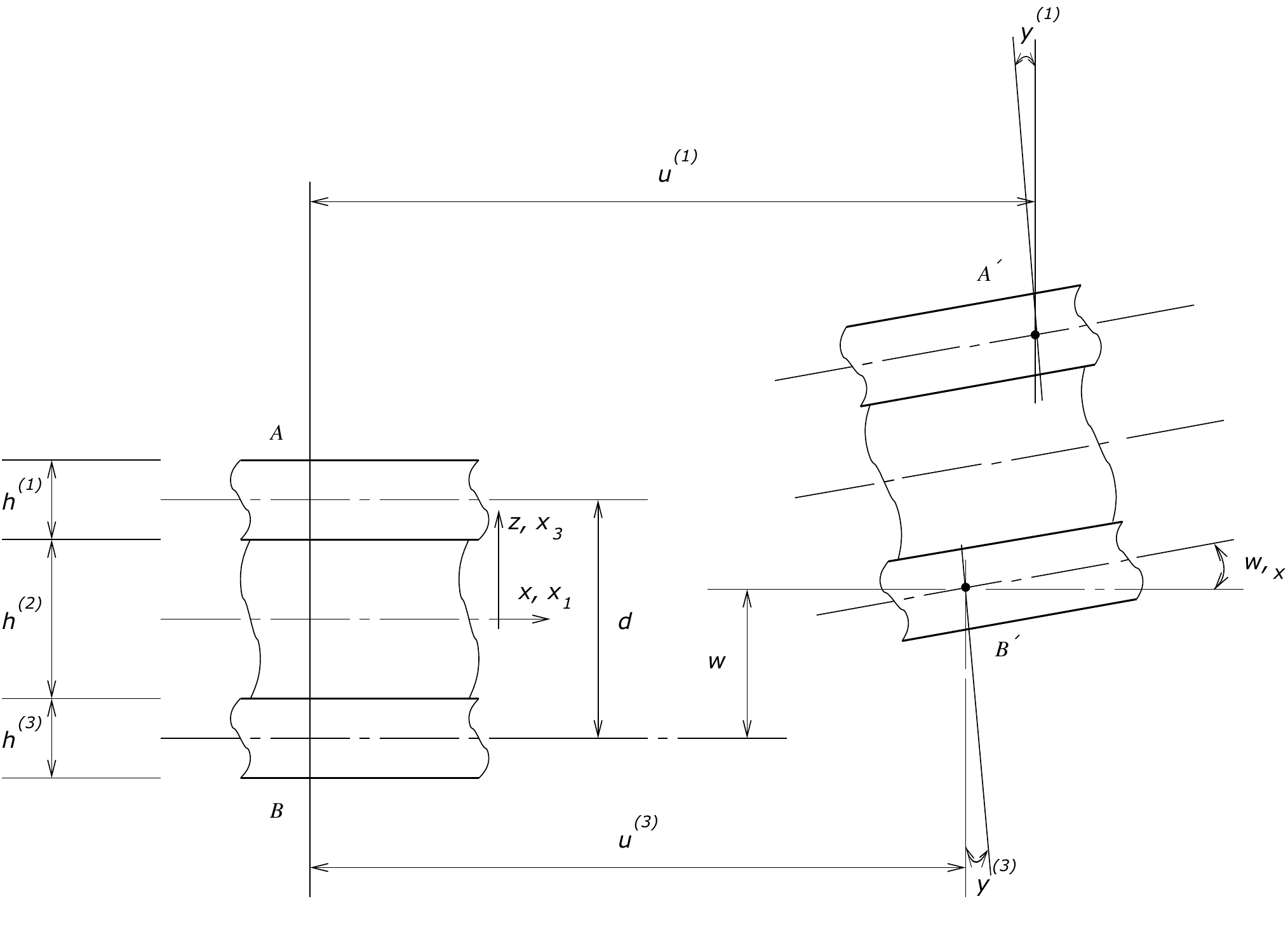}
\caption{.}
\end{center}
\end{figure}
\newpage \noindent When the rotatory inertia and the transverse shear of the top and bottom layers are neglected, equations \eqref{1LB-1}$_2$ and \eqref{1LB-1}$_5$ reduce to the familiar Euler-Bernoulli hypothesis $y^1-\omega_x=y^3-\omega_x=0$. If we consider the core material to be linearly elastic, i.e. $\tau=2G_2\gamma$, with the shear strain 
$$
\gamma=\frac{1}{2h_2}\left(-u^1+u^3+\alpha\omega_x\right),
$$
and $\alpha=h_2+\frac{1}{2}\left(h_1++h_3\right)$, we then obtained the Rao-Nakra model \cite{SADASIVARAO1974309},
\begin{equation}\label{Rao-Nakra}
\left\{\begin{array}{l}
\rho_1h_1u^1_{tt}-E_1h_1u^1_{xx}-\frac{G_2}{h_2}\left(-u^1+u^3+\alpha\omega_x\right)=0,\\[0.1in]
\rho_3h_3u^3_{tt}-E_3h_3u^3_{xx}+\frac{G_2}{h_2}\left(-u^1+u^3+\alpha\omega_x\right)=0,\\[0.1in]
\rho h\omega_{tt}+EI\omega_{xxxx}-\frac{G_2\alpha}{h_2}\left(-u^1+u^3+\alpha\omega_x\right)_x=0.
\end{array}
\right.
\end{equation}
Furthermore, if the extensional forces in the top and bottom layers are also neglected, we obtain the Mead-Markos model (see \cite{MEAD1969163}), 
\begin{equation}\label{Mead-Markus}
\left\{\begin{array}{l}
\rho h\omega_{tt}+EI\omega_{xxxx}-G_2\alpha\gamma=0,\\
2h_2\gamma_{xx}-G_2d\gamma-\alpha \omega_{xxx}=0.
\end{array}
\right.
\end{equation}
which can be simplified into a six-order PDE for $\omega$. In \cite{LIU1999149}, exponential stability was proved for the Mead-Markus model \eqref{Mead-Markus}, when the shear stress $\tau$ and shear strain $\gamma$ relation is assumed to be viscoelastic of Boltzmann type. When this relationship is of Kelvin-Voigt type, analyticity of the associated semigroup was proved by Hansen and Liu (see \cite{Hansen1999}), which was further extended to the corresponding multi-layers beam and plate model by Allen and Hansen in \cite{ALLEN2009e1835} and \cite{Hansen-ALLen}.\\[0.1in] 
When the extensional motion of the bottom and top layers is neglected, we obtain the model proposed by Hansen and Spies (see \cite{HANSEN1997183}). For the model in \cite{HANSEN1997183}, exponential stability was proved in \cite{doi:10.1137/040610003} when structural damping and boundary damping are added, or when viscous damping are added to all three equation (see \cite{RAPOSO201685}).\\[0.1in]
In 2013, (see \cite{HansenOZER}), { exponential stability of Rao-Nakra model \eqref{Rao-Nakra} was obtained when standard boundary damping are imposed on one end of the beam for all three layers}.\\[0.1in]
The boundary controllability problems of the Rao-Nakra beam equation (multi-layers, $\alpha>0$) have been studied in a series of papers \cite{refId0,HansenOZER,RAJARAM2007558,1582645}. In \cite{refId0}, exact controllability results for the { multi-layers} Rao-Nakra plate system with locally distributed control in a neighborhood of a portion of the boundary were obtained by the method of {Carleman} estimates.\\[0.in] 
{In 2018, (see \cite{LiLiuWang})  the authors considered the Rao-Nakra system (1.2) with viscous and/or Kelvin-Voigt dampings. They
first showed that the system is unstable when only the transverse displacement is damped and $\frac{E_1}{\rho_1}=\frac{E_3}{\rho_3}$. Moreover, they considered seven cases of the combination of damping locations and types when two displacements are damped. Polynomial stability of different orders and their optimality are obtained for a particular type boundary conditions. In 2020, (see \cite{LIU20206125}),  the authors considered the stability of the Rao-
Nakra sandwich beam equation with various boundary conditions. Polynomial stability of certain orders are
obtained when there is only one viscous damping acting either on the beam equation or one of the wave equations. For a few special cases, optimal orders are confirmed. They also studied the synchronization of
the model with viscous damping on the transversal displacement. Their results reveal that the order of the polynomial decay rate is sensitive to various boundary conditions and to the damping locations. }
\\[0.1in]
In this paper, we consider  the {generalized Rao-Nakra beam model } with only two viscous damping  described by 
\begin{equation}\label{LB-1}
\left\{\begin{array}{lll}
\rho_1h_1u^1_{tt}-E_1h_1u^1_{xx}-\tau+au^1_t=0&\text{in}&(0,L)\times \R^{+},\\[0.1in]
\rho_1I_1y^1_{tt}-E_1I_1y^1_{xx}-\frac{h_1}{2}\tau+G_1h_1(\omega_x+y^1)+by^1_t=0&\text{in}&(0,L)\times \R^{+},\\[0.1in]
\rho h\omega_{tt}+EI\omega_{xxxx}-G_1h_1(\omega_x+y^1)_x-G_3h_3(\omega_x+y^3)_x-h_2\tau_x+c\omega_t=0&\text{in}&(0,L)\times \R^{+},\\[0.1in]
\rho_3h_3u^3_{tt}-E_3h_3u^3_{xx}+\tau+du^3_t=0&\text{in}&(0,L)\times \R^{+},\\[0.1in]
\rho_3I_3y^3_{tt}-E_3I_3y^3_{xx}-\frac{h_3}{2}\tau+G_3h_3(\omega_x+y^3)+ey^3_t=0&\text{in}&(0,L)\times \R^{+},
\end{array}
\right.
\end{equation}
where $a,b,c,d,e\geq 0$ be dissipative coefficients. System \eqref{LB-1} is subjected to the following initial conditions 
\begin{equation}\label{LB-2}
\left\{\begin{array}{l}
u^1(x,0)=u_0^1(x),\ u^1_t(x,0)=u_1^1(x),\\[0.1in]
u^3(x,0)=u_0^2(x),\ u^3_t(x,0)=u_1^3(x),\\[0.1in]
\omega(x,0)=\omega_0(x),\ \omega_t(x,0)=\omega_1(x),\\[0.1in]
y^1(x,0)=y_0^1(x),\ y^1_t(x,0)=y_1^1(x),\\[0.1in]
y^3(x,0)=u_0^2(x),\ y^3_t(x,0)=u_1^3(x),
\end{array}
\right.
\end{equation}
For the boundary conditions, we assume that variables $u^1,u^3,y^1$ and $y^3$ satisfy {the} Dirichlet boundary conditions
\begin{equation}\label{LB-3}
u^1(0,t)=u^1(L,t)=u^3(0,t)=u^3(L,t)=y^1(0,t)=y^1(L,t)=y^3(0,t)=y^3(L,t)=0,
\end{equation}
and $\omega$ satisfies the {clamped} boundary condition
\begin{equation}\label{FB}
\omega(0,t)=\omega(L,t)=\omega_x(0,t)=\omega_x(L,t)=0.
\end{equation}
We shall investigate the stability of system \eqref{LB-1}-\eqref{FB} with only two viscous damping. The location of the damping are divided into two groups, characterized by whether both of the top and bottom layers are directly damped or otherwise. Group one consists of three cases, i.e., $a,d\ne 0, b=c=e=0$, or $b,e\ne 0, a=c=d=0$, or $a,e\ne 0, b=c=d=0$. Group two also consists of three cases, i.e., $a,b\ne 0, c=d=e=0$, or $a,c\ne 0, b=d=e=0$, or $a,c\ne 0, b=d=e=0$. \\[0.1in]
This paper is organized as follows. In section \ref{WPSS}, we present the semigroup setting of the system for well-posedness, including the necessary and sufficient conditions for strong stability of the cases in group two, and { sufficient conditions for strong stability or instability for some cases in group one}. Section \ref{PS123} is devoted to the polynomial stability { of the} three cases in group two. A concluding remark is given in the end of that section.

\section{Well-Posedness and Strong Stability}\label{WPSS}
Let $(u^1,u^3,\omega,y^1,y^3)$ be a regular solution of \eqref{LB-1}-\eqref{LB-3}, its associated energy is given by 
\begin{equation}\label{Energy-RN}
\begin{array}{lcl}
E(t)&=&\displaystyle
\frac{1}{2}\left(\rho_1h_1\int_0^L\abs{u_t}^2dx+E_1h_1\int_0^L\abs{u_x^1}^2dx+\rho_3h_3\int_0^L\abs{u_t^3}^2dx+E_3h_3\int_0^L\abs{u^3_x}^2dx\right.\\[0.1in]
&&\displaystyle
\left.+\rho h\int_0^L\abs{\omega_t}^2dx+EI\int_0^L\abs{\omega_{xx}}^2dx+\rho_1I_1\int_0^L\abs{y_t^1}^2dx+E_1I_1\int_0^L\abs{y_x^1}^2dx\right.\\[0.1in]
&&\displaystyle
\left.+\rho_3I_3\int_0^L\abs{y^3_t}^2dx+E_3I_3\int_0^L\abs{y^3_x}^2dx+G_1h_1\int_0^L\abs{\omega_x+y^1}^2dx+\int_0^L\abs{\tau}^2dx\right.\\[0.1in]
&&\displaystyle
\left.+G_3h_3\int_0^L\abs{\omega_x+y^3}^2dx\right). 
\end{array}
\end{equation}
A direct computation gives that  
\begin{equation}\label{dEnergy-RN}
\frac{d}{dt}E(t)=-a\int_0^L\abs{u^1_t}^2dx-b\int_0^L\abs{y^1_t}^2dx-c\int_0^L\abs{\omega_t}^2dx-d\int_0^L\abs{u^3_t}^2-e\int_0^L\abs{y_t^3}^2dx\leq 0
\end{equation}
Now, we define the following energy space 
$$
\mathcal{H}=\left(H_0^1(0,L)\times L^2(0,L)\right)^2\times H_0^2(0,L)\times L^2(0,L)\times \left(H_0^1(0,L)\times L^2(0,L)\right)^2. 
$$
equipped with the inner product which induces the energy norm:
\begin{equation}\label{norm}
\begin{array}{lll}
\|U\|_{\mathcal{H}_{j}}^2&=&\|\left({u^1,v^1,y^1,z^1,\omega,\psi,u^3,v^3,y^3,z^3}\right)\|^1_{\HH_j}\\[0.1in]
&=&E_1h_1\|u_x^1\|^2+\rho_1h_1\|v^1\|^2+E_3h_3\|u^3_x\|^2+\rho_3h_3\|v^3\|^2dx+EI\|\omega_{xx}\|^2+\rho h\|\psi\|^2\\[0.1in]
&&+E_1I_1\|y^1_x\|^2+\rho_1I_1\|z^1\|^2+E_3I_3\|y^3_x\|^2+\rho_3I_3\|z^3\|^2+G_1h_1\|\omega_x+y^1\|^2\\[0.1in]
&&+\|\tau\|^2+G_3h_3\|\omega_x+y^3\|^2.
\end{array}
\end{equation}
It is clear that the above equality is {an} equivalent norm  on $\mathcal{H}$. Now, we define an unbounded linear {operator} $\mathcal{A}:D(\mathcal{A})\subset { \mathcal{H}\rightarrow \mathcal{H}}$, by 
\begin{equation}\label{Operators}
\mathcal{A}\begin{pmatrix}u^1\\ v^1\\ y^1\\ z^1\\ \omega\\ \psi\\ u^3\\ v^3\\ y^3\\ z^3\end{pmatrix}=\begin{pmatrix}
v^1\\ (\rho_1h_1)^{-1}\left[E_1h_1u^1_{xx}+\tau -a v^1\right]\\ z^1\\ (\rho_1I_1)^{-1}\left[E_1I_1y^1_{xx}+\frac{h_1}{2}\tau-G_1h_1(\omega_x+y^1)-bz^1\right]\\ \psi\\ (\rho h)^{-1}\left[-EI\omega_{xxxx}+G_1h_1(\omega_x+y^1)_x+G_3h_3(\omega_x+y^3)_x+h_2\tau_x-c\psi\right]\\ v^3\\ (\rho_3 h_3)^{-1}\left[E_3h_3u^3_{xx}-\tau-dv^3\right]\\ z^3\\ (\rho_3I_3)^{-1}\left[E_3I_3y_{xx}^3+\frac{h_3}{2}\tau-G_3h_3(\omega_x+y^3)-ez^3\right] 
\end{pmatrix}
\end{equation}
with domain 
$$
D(\mathcal{A})=\left[\left(H^2(0,L)\cap H_0^1(0,L)\right)\times L^2(0,L)\right]^2\times (H^4(0,L)\cap H_0^2(0,L))\times L^2(0,L)\times \left[\left(H^2(0,L)\cap H_0^1(0,L)\right)\times L^2(0,L)\right]^2.
$$
If $U=(u^1,u^1_t,y^1,y^1_t,\omega,\omega_t,u^3,u^3_t,y^3,y^3_t)$ is a regular solution of system \eqref{LB-1}-\eqref{LB-3}, then the system can be rewritten as { an} evolution equation on the Hilbert space $\mathcal{H}$ given by 
\begin{equation}\label{evolution}
U_t=\mathcal{A} U,\quad U(0)=U_0,
\end{equation}
where $U_0=(u^1_0,u^1_1,y^1_0,y^1_1,\omega_0,\omega_1,u^3_0,u^3_1,y^3_0,y^3_1)$. 
It is easy to see that for all $U=\left(u^1,v^1,y^1,z^1,\omega,\psi,u^3,v^3,y^3,z^3\right)\in D(\mathcal{A})$, {we have}  
$$
\Re\left(\left<\mathcal{A}U,U\right>\right)=-a\int_0^L\abs{v^1}^2dx-b\int_0^L\abs{z^1}-c\int_0^L\abs{\psi}^2dx-d\int_0^L\abs{v^3}^2dx-e\int_0^L\abs{z^3}^2dx\leq 0,
$$
which implied that $\mathcal{A}$ is dissipative. Now, let $F=(f_1,\cdots,f_{10})\in \mathcal{H}$, by using the Lax-Milgram theorem it is easy to show that the existence of $U\in D(\mathcal{A})$, solution of the equation 
$$
-\mathcal{A}U=F.
$$
Then, the unbounded linear operator $\mathcal{A}$ is m-dissipative in the energy space $\mathcal{H}$ and consequently $0\in \rho(\mathcal{A})$. {Thus, $\mathcal{A}$} generates a $C_0-$semigroup of contractions $\left(e^{t\mathcal{A}}\right)_{t\geq 0}$ following {the} Lummer-Phillips theorem. {The} solution of the {Cauchy problem} \eqref{evolution} admits the following representation 
$$
U(t)=e^{t\mathcal{A}}U_0,\quad t\geq 0,
$$
which leads to the well-posedness of \eqref{evolution}. Hence, we have the following result. 
\begin{theoreme}
Let $U_0\in \mathcal{H}$, then {system} \eqref{evolution} admits a unique weak solution $U$ satisfies 
$$
U \in C^0\left(\mathbb{R}^+,\mathcal{H}\right).  
$$
Moreover, if $U_0\in D(\mathcal{A})$, then problem \eqref{evolution} admits a unique strong solution $U$ satisfies 
$$
U\in C^1\left(\mathbb{R}^+,\mathcal{H}\right)\cap C^0\left(\mathbb{R}^+,D(\mathcal{A})\right).
$$
\end{theoreme}
$\linebreak$
Now, we shall analyze the strong stability of system \eqref{evolution}. The main result of this section is the following theorem. 
\begin{theoreme}\label{strong}
Assume that $(a,b>0\ \text{and}\ c=e=d=0)$ or $(a,c>0\ \text{and}\ b=e=d=0\ \text{and}\ G_1\neq G_3)$ or $(b,c>0\ \text{and}\ a=e=d=0\ \text{and}\ \frac{E_1}{\rho_1}\neq \frac{E_3}{\rho_3})$ or $(b,e>0\ \text{and}\ a=c=d=0\ \text{and}\ (G_1\neq G_3\ \text{and}\ \frac{E_1}{\rho_1}\neq \frac{E_3}{\rho_3}))$ or $(a,e>0\ \text{and}\ b=c=d=0)$. Then, the $C_0-$semigroup of contractions $(e^{t\mathcal{A}})_{t\geq 0}$ is strongly stable on $\mathcal{H}$ in the sense that $\displaystyle{\lim_{ { t}\to +\infty}\|e^{t\mathcal{A}}U_0\|_{\mathcal{H}}}=0$. 
\end{theoreme}
\begin{proof}
Since the resolvent of $\mathcal{A}$ is compact in $\mathcal{H}$, then according to Arendt-Batty theorem see (Page 837 in \cite{Batty01}), system \eqref{LB-1}-\eqref{FB} is strongly stable if and only if $\mathcal{A}$ doesn't have pure imaginary eigenvalues, that is, $\sigma(\mathcal{A})\cap i\mathbb{R}=\emptyset$. We have {already shown that $0\in \rho(\mathcal{A})$, and} still need to show $\sigma(\mathcal{A})\cap i\mathbb{R^{\ast}}=\emptyset$ . For this aim, suppose by contradiction that there exists  $\la\in \mathbb{R}^{\ast}$ and $U\in D(\mathcal{A})\backslash \{0\} $ such that 
\begin{equation}\label{LB-st1}
\mathcal{A}U=i\la U.
\end{equation}
Equivalently, we have 
\begin{equation}\label{LB-st2}
v^1=i\la u^1,\ v^3=i\la u^3,\ \psi=i\la \omega,\ z^1=i\la y^1 \quad \text{and} \quad z^3=i\la y^3 
\end{equation}
and 
\begin{eqnarray}
\la^2u^1+(\rho_1h_1)^{-1}\left[E_1h_1u^1_{xx}+\tau -ia\la u^1\right]&=&0,\label{1LB-st}\\[0.1in]
\la^2 y^1+(\rho_1I_1)^{-1}\left[E_1I_1y^1_{xx}+\frac{h_1}{2}\tau-G_1h_1(\omega_x+y^1)-ib\la y^1\right]&=&0,\label{2LB-st}\\[0.1in]
\la^2\omega+(\rho h)^{-1}\left[-EI\omega_{xxxx}+G_1h_1(\omega_x+y^1)_x+G_3h_3(\omega_x+y^3)_x+h_2\tau_x-ic\la \omega\right]&=&0,\label{3LB-st}\\[0.1in]
\la^2u^3+(\rho_3 h_3)^{-1}\left[E_3h_3u^3_{xx}-\tau-id\la u^3\right]&=&0,\label{4LB-st}\\[0.1in]
\la^2y^3+(\rho_3I_3)^{-1}\left[E_3I_3y_{xx}^3+\frac{h_3}{2}\tau-G_3h_3(\omega_x+y^3)-ie\la y^3\right]&=&0\label{5LB-st}.
\end{eqnarray}
\textbf{Case 1.} Suppose that ($a,b>0$ and $c=0$).
A straightforward calculation gives 
$$
{ 0=\Re\left<i\la U,U\right>_{\mathcal{H}}=\Re\left<\mathcal{A}U,U\right>_{\mathcal{H}}=-a\|v^1\|^2-b\|z^1\|^2}.
$$
Consequently, we deduce that 
\begin{equation}\label{LB-st4}
v^1=z^1=u^1=y^1=0\quad \text{in}\quad (0,L).
\end{equation}
Using \eqref{LB-st4} and \eqref{1LB-st} and the definition of $\tau$, we get 
\begin{equation}\label{LB-st5}
\tau=u^3+h_2\omega_x-\frac{h^3}{2}y^3=0.
\end{equation}
Inserting \eqref{LB-st5} and  \eqref{LB-st4} in  \eqref{2LB-st} and by using the fact that  $\omega(0)=0$, we obtain 
\begin{equation}\label{LB-st6}
\omega=0\quad \text{in}\quad (0,L). 
\end{equation}
Using \eqref{LB-st4}-\eqref{LB-st6} in  \eqref{3LB-st}, and using the fact that $y^3(0)=0$, we get 
\begin{equation}\label{LB-st7}
y^3=0.
\end{equation}
Equations \eqref{LB-st5}-\eqref{LB-st7}, implies that 
\begin{equation*}
u^3=0.
\end{equation*}
Thus, $U=0$ and consequently $\mathcal{A}$ has no pure imaginary eigenvalues.\\
\textbf{Case 2.} Suppose that ($a,c>0$ and $b=0$ and $G_1\neq G_3$). A straightforward calculation gives 
$$
{ 0=\Re\left<i\la U,U\right>_{\mathcal{H}}=\Re\left<\mathcal{A}U,U\right>_{\mathcal{H}}=-a\|v^1\|^2-c\|\psi\|^2}.
$$
Consequently, we deduce that 
\begin{equation}\label{LB-Case2-st1}
v^1=u^1=\psi=\omega=0\quad \text{in}\quad (0,L).
\end{equation}
Using \eqref{LB-Case2-st1},  \eqref{1LB-st} and the definition of $\tau$, we get 
\begin{equation}\label{LB-Case2-st2}
\tau=u^3-\frac{h_1}{2}y^1-\frac{h_3}{2}y^3=0.
\end{equation}
Using  \eqref{3LB-st}, \eqref{LB-Case2-st1}, \eqref{LB-Case2-st2} and the fact that $y^1(0)=y^3(0)=0$, we get 
\begin{equation}\label{LB-Case2-st3}
y^1=-\frac{G_3h_3}{G_1h_1}y^3\quad \text{and}\quad u^3=\frac{h_3}{2G_1}\left(G_1-G_3\right)y^3. 
\end{equation}
Using \eqref{4LB-st}, \eqref{5LB-st} and \eqref{LB-Case2-st2},\eqref{LB-Case2-st3}, the fact that $G_1\neq G_3$ we get 
\begin{equation*}
\la^2y^3+\frac{E_3}{\rho_3}y^3_{xx}=0\quad \text{and}\quad \left(\la^2-\frac{G_3h_3}{\rho_3I_3}\right)y^3+\frac{E_3}{\rho_3}y_{xx}^3=0
\end{equation*}
Subtract the above two equations, we obtain 
\begin{equation}\label{LB-Case2-st4}
y^3=0.
\end{equation}
Finally, using \eqref{LB-Case2-st4} and \eqref{LB-Case2-st3}, we get $y^1=u^3=0$. Thus, $U=0$ and consequently $\mathcal{A}$ has no pure imaginary eigenvalues.\\
\textbf{Case 3.} Suppose that $(b,c>0\ \text{and}\ a=0\ \text{and}\ \frac{E_1}{\rho_1}\neq \frac{E_3}{\rho_3})$. A straightforward calculation gives
\begin{equation}\label{LB-Case3-st1}
z^1=y^1=\omega=\psi=0\quad \text{in}\quad (0,L).
\end{equation}
Using \eqref{LB-Case3-st1} and  \eqref{2LB-st} and the definition of $\tau$, we get 
\begin{equation}\label{LB-Case3-st2}
\tau=-u^1+u^3-\frac{h_3}{2}y^3=0.
\end{equation}
Using \eqref{LB-Case3-st1}, \eqref{LB-Case3-st2} in  \eqref{3LB-st} and the fact that $y^3(0)=0$, we get 
\begin{equation}\label{LB-Case3-st3}
y^3=0.
\end{equation}
Inserting \eqref{LB-Case3-st3} in \eqref{LB-Case3-st2}, we obtain 
\begin{equation}\label{LB-Case3-st4}
u^1=u^3.
\end{equation}
Now, using \eqref{LB-Case3-st4} and  \eqref{LB-Case3-st2} in  \eqref{1LB-st} and \eqref{4LB-st}, we get 
\begin{equation}\label{LB-Case3-st5}
\la^2u^3+\frac{E_1}{\rho_1}u^3_{xx}=0\quad \text{and}\quad \la^2u^3+\frac{E_1}{\rho_1}u^3_{xx}=0.
\end{equation}
Subtract the above two equations, we get 
\begin{equation*}
\left(\frac{E_1}{\rho_1}-\frac{E_3}{\rho_3}\right)u^3_{xx}=0.
\end{equation*}
Using the fact that $\frac{E_1}{\rho_1}\neq \frac{E_3}{\rho_3}$ and the facts that $u^3(0)=u^3(L)=0$, we get 
$$
u^3=0.
$$
Thus, $U=0$ and consequently $\mathcal{A}$ has no pure imaginary eigenvalues.\\
\textbf{Case 4.} Suppose that $(b,e>0\ \text{and}\ a=c=d=0\ \text{and}\ (G_1\neq G_3\ \text{and}\ \frac{E_1}{\rho_1}\neq \frac{E_3}{\rho_3}))$. A straightforward calculation gives 
\begin{equation}\label{4be}
z^1=y^1=z^3=y^3=0\quad \text{in}\quad (0,L).
\end{equation}
From \eqref{2LB-st},\eqref{5LB-st} and \eqref{4be}, we get 
\begin{equation}\label{4be1}
\frac{1}{2}\tau -G_1\omega_x=0\quad \text{and}\quad \frac{1}{2}\tau-G_3\omega_x=0.
\end{equation}
Subtracting the two equations in \eqref{4be1} and using the fact that $G_3\neq G_1$, we get 
\begin{equation}\label{4be2}
\omega_x=0.
\end{equation}
Using \eqref{4be1} and  the fact that $\omega(0)=0$, we get 
\begin{equation}\label{4be3}
\omega=0\quad \text{and}\quad \tau=0.
\end{equation}
Using the definition of $\tau$ and \eqref{4be} and \eqref{4be3}, we get 
\begin{equation}\label{4be4}
u^1=u^3.
\end{equation}
Now, using \eqref{1LB-st}, \eqref{4LB-st} and \eqref{4be3}, we get 
\begin{equation*}
\la^2u^1+\frac{E_1}{\rho_1}u^1_{xx}=0\quad \text{and}\quad \la^2u^3+\frac{E_3}{\rho_3}u^3_{xx}=0.
\end{equation*}
Subtracting the above two equations and using \eqref{4be4}, we get 
\begin{equation}\label{4be5}
\left(\frac{E_1}{\rho_1}-\frac{E_3}{\rho_3}\right)u^1_{xx}=0.
\end{equation}
Using the facts that $\frac{E_1}{\rho_1}\neq \frac{E_3}{\rho_3}$ and $u^1(0)=u^1(L)=0$ in \eqref{4be5}, we obtain 
$$
u^1=u^3=0.
$$
Thus, $U=0$ and consequently $\mathcal{A}$ has no pure imaginary eigenvalues.\\[0.1in]
\textbf{Case 5.} Suppose that $(a,e>0\ \text{and}\ b=c=d=0)$. A straightforward calculation gives 
$$
0=\Re\left(\left<i\la U,U\right>_{\mathcal{H}}\right)=\Re\left(\left<\mathcal{A}U,U\right>_{\mathcal{H}}\right)=-a\|v^1\|^2-e\|z^3\|^2=0.
$$
Consequently, we deduce that 
\begin{equation}\label{5LB-st4}
v^1=z^3=u^1=y^3=0\quad \text{in}\quad (0,L).
\end{equation}
Using \eqref{5LB-st4} and \eqref{1LB-st} and the definition of $\tau$, we get 
\begin{equation}\label{5LB-st5}
\tau=u^3+h_2\omega_x-\frac{h^1}{2}y^1=0.
\end{equation}
Inserting \eqref{5LB-st5} and  \eqref{5LB-st4} in  \eqref{5LB-st} and by using the fact that  $\omega(0)=0$, we obtain 
\begin{equation}\label{5LB-st6}
\omega=0\quad \text{in}\quad (0,L). 
\end{equation}
Using \eqref{5LB-st4}-\eqref{5LB-st6} in  \eqref{3LB-st}, and using the fact that $y^1(0)=0$, we get 
\begin{equation}\label{5LB-st7}
y^1=0.
\end{equation}
Equations \eqref{5LB-st5}-\eqref{5LB-st7}, implies that 
\begin{equation*}
u^3=0.
\end{equation*}
Thus, $U=0$ and consequently $\mathcal{A}$ has no pure imaginary eigenvalues.
\end{proof}
\begin{theoreme}\label{INSTABLE1}
If $(a,c>0\ \text{and}\ b=d=e=0)$ and $G_1=G_3$, then system \eqref{evolution} is unstable.
\end{theoreme}
\begin{proof}
Set 
\begin{equation}\label{unstable1}
u^1(x,t)=u^3(x,t)=\omega(x,t)=0, y^3(x,t)=e^{i\la_n t}\phi^n(x)\quad \text{and}\quad  y^1(x,t)=-\frac{h_3}{h_1}y^3(x,t).
\end{equation}
Substituting \eqref{unstable1} in \eqref{evolution}, we obtain that 
\begin{equation}\label{unstable2}
\left\{\begin{array}{l}
\left(\la^2-\frac{G_3h_3}{\rho_3I_3}\right)\phi^n+\frac{E_3}{\rho_3}\phi^n_{xx}=0,\\
\phi(0)=\phi(L)=0.
\end{array}
\right.
\end{equation}
Then, choosing $\la_n=\pm\sqrt{\frac{n^2\pi^2}{L^2}\frac{E_3}{\rho_3}+\frac{G_1h_3}{\rho_3I_3}}$, we obtain 
$$
\phi(x)=A \sin\left(\frac{n\pi}{L}x\right).
$$
{This} implies that there are infinitely many eigenvalues $\pm i\la_n$ on the imaginary axis. 
\end{proof}
\begin{theoreme}\label{INSTABLE2}
If $\left((b,c>0\ \text{and}\ a=d=e=0)\ \text{and}\  \frac{E_1}{\rho_1}=\frac{E_3}{\rho_3}\right)$ or ${\left(b,e>0\ \text{and}\ a=c=d=0\ \text{and}\ (G_1\neq G_3\ \text{and}\ \frac{E_1}{\rho_1}= \frac{E_3}{\rho_3})\right)}$ or ${\left(b,e>0\ \text{and}\ a=c=d=0\ \text{and}\ (G_1= G_3\ \text{and}\ \frac{E_1}{\rho_1}= \frac{E_3}{\rho_3})\right)}$ , then system \eqref{evolution} is unstable.
\end{theoreme}
\begin{proof}
Set 
\begin{equation}\label{UNSTABLE-Case3}
u^1(x,t)=u^3(x,t)=e^{i\la_nt}\phi^n(x)\quad \text{and}\quad y^1(x,t)=y^3(x,t)=\omega(x,t)=0.
\end{equation}
Substituting \eqref{UNSTABLE-Case3} in \eqref{evolution}, we obtain that 
\begin{equation}\label{UNSTABLE-Case3-1}
\left\{\begin{array}{l}
\displaystyle
\la^2\phi^n+\frac{E_3}{\rho_3}\phi^n_{xx}=0,\\
\phi^n(0)=\phi^n(L)=0.
\end{array}
\right.
\end{equation}
Then, choosing $\la_n=\pm\sqrt{\frac{E_3}{\rho_3}}\frac{n\pi}{L}$, we obtain 
$$
\phi^n(x)=B\sin\left(\frac{n\pi x}{L}\right).
$$ 
This implies that there are infinitely eigenvalues $\pm i\la_n$ on the imaginary axis.
\end{proof}
\begin{theoreme}
If $(a,d>0\ \ \text{and}\ \ b=c=e=0)$ and $E_1=E_3$, $\rho_1=\rho_3$, $G_1=G_3$\ $I_1=I_3$ and $h_1=h_3$, then system \eqref{evolution} is unstable
\end{theoreme}
\begin{proof}
Set 
\begin{equation}\label{UNSTABLE-Case4}
u^1(x,t)=u^3(x,t)=\omega(x,t)=0,\ \ y^1(x,t)=e^{i\la_n}\phi^n(x)\quad \text{and}\quad y^3(x,t)=-y^1(x,t).
\end{equation}
Substituting \eqref{UNSTABLE-Case3} in \eqref{evolution}, we obtain that 
\begin{equation}\label{UNSTABLE-Case4-1}
\left\{\begin{array}{l}
\left(\la^2-\frac{G_1h_1}{\rho_1I_1}\right)\phi^n+\frac{E_1}{\rho_1}\phi^n_{xx}=0,\\
\phi(0)=\phi(L)=0.
\end{array}
\right.
\end{equation}
Then, choosing $\la_n=\pm\sqrt{\frac{n^2\pi^2}{L^2}\frac{E_1}{\rho_1}+\frac{G_1h_1}{\rho_1I_1}}$, we obtain 
$$
\phi(x)=C \sin\left(\frac{n\pi}{L}x\right).
$$
This implies that there are infinitely many eigenvalues $\pm i\la_n$ on the imaginary axis. 

\end{proof}
\begin{rem}
{  The first three cases of Theorem \ref{strong},  Theorem \ref{INSTABLE1} and the first case in Theorem \ref{INSTABLE2} give the necessary and sufficient conditions for the strong stability of group two. However, we only obtained sufficient conditions for strong stability and instability of some cases of group one.  
Thus, we will only consider the polynomial stability for the cases in group two in next section. }
\end{rem}
\section{Polynomial Stability}\label{PS123}
\noindent In this section, we show the influence of the physical coefficients on the stability of system \eqref{LB-1}-\eqref{FB} with respect to the  {locations} of the damping. For this aim, we distinguish the following three cases:\begin{enumerate}
\item[$\bullet$] \textbf{Case 1}:
\begin{equation}\label{A1} \tag{${\rm A_1}$}
(a,b>0\ \text{and}\ c=d=e=0). 
\end{equation}
\item[$\bullet$] \textbf{Case 2}: \begin{equation}\label{A2} \tag{${\rm A_2}$}(a,c>0\ \text{and}\ b=e=d=0\ \text{and}\ G_1\neq G_3).\end{equation}
\item[$\bullet$] \textbf{Case 3}: \begin{equation}\label{A3} \tag{${\rm A_3}$}(b,c>0\ \text{and}\ a=e=d=0\ \text{and}\ \frac{E_1}{\rho_1}\neq \frac{E_3}{\rho_3}).\end{equation}
\end{enumerate}
For this purpose, we will use a frequency domain {approach, namely the Theorem A.3 in the Appendix}.  Our main result in this section is the following theorems. 
\begin{theoreme}\label{POL-ab}
Assume that \eqref{A1} holds. The $C_0-$semigroup $\left(e^{t\mathcal{A}}\right)_{t\geq1}$ is polynomially stable; i.e. there exists constant $C>0$ such that for every $U_0\in D(\mathcal{A})$, we have 
\begin{equation}\label{pol-decay-ab}
E(t)\leq \frac{C}{t^{\frac{2}{\ell}}},\quad t>0,\ \forall U_0\in D(\mathcal{A}),
\end{equation}
where 
$$
\ell=\left\{\begin{array}{lll}
3&\text{if}&\frac{E_1}{\rho_1}=\frac{E_3}{\rho_3},\\[0.1in]
5&\text{if}&\frac{E_1}{\rho_1}\neq \frac{E_3}{\rho_3}.
\end{array}
\right.
$$
\end{theoreme}
\begin{theoreme}\label{POL-ac}
Assume that \eqref{A2} holds. The $C_0-$semigroup $\left(e^{t\mathcal{A}}\right)_{t\geq1}$ is polynomially stable; i.e. there exists constant $C>0$ such that for every $U_0\in D(\mathcal{A})$, we have 
\begin{equation}\label{pol-decay-ac}
E(t)\leq \frac{C}{t^{\frac{2}{\ell}}},\quad t>0,\ \forall U_0\in D(\mathcal{A}),
\end{equation}
where 
$$
\ell=\left\{\begin{array}{lll}
2&\text{if}&\left(\frac{E_1}{\rho_1}=\frac{E_3}{\rho_3}\quad\text{and}\quad \frac{G_3h_3}{\rho_3I_3}\neq \frac{G_1h_1}{\rho_1I_1}\right),\\[0.1in]
6&\text{if}&\left(\frac{E_1}{\rho_1}\neq \frac{E_3}{\rho_3}\ \quad \text{or}\quad \left(\frac{E_1}{\rho_1}=\frac{E_3}{\rho_3}\quad\text{and}\quad \frac{G_3h_3}{\rho_3I_3}= \frac{G_1h_1}{\rho_1I_1}\right)\right).
\end{array}
\right.
$$
\end{theoreme}
\begin{theoreme}\label{POL-bc}
Assume that \eqref{A3} holds. The $C_0-$semigroup $\left(e^{t\mathcal{A}}\right)_{t\geq1}$ is polynomially stable; i.e. there exists constant $C>0$ such that for every $U_0\in D(\mathcal{A})$, we have 
\begin{equation}\label{pol-decay-ac}
E(t)\leq {\frac{C}{t^{\frac{1}{3}}}},\quad t>0,\ \forall U_0\in D(\mathcal{A}),
\end{equation}
\end{theoreme}
\noindent According to Theorem \ref{bt}, the polynomial energy decay {\eqref{pol-decay-ab}-\eqref{pol-decay-ac} hold} if the following conditions 
\begin{equation}\tag{$\rm{H_1}$}\label{H1}
i\mathbb{R}\subset \rho(\mathcal{A}),
\end{equation}
and 
\begin{equation}\tag{$\rm{H_2}$}\label{H22}
\sup_{\la\in\mathbb{R}}\|(i\la I-\mathcal{A})^{-1}\|_{\mathcal{L}(\mathcal{H})}{\le}O\left(\abs{\la}^{\ell}\right)
\end{equation}
are satisfied. Since condition \eqref{H1} is already proved in Theorem \ref{strong}. We will prove condition \eqref{H22} by an argument of contradiction. For this purpose, suppose that {it} is false, then there exists
$$
\left\{(\la_n,U_n:=(u^1,v^1,y^1,z^1,\omega,\psi,u^3,v^3,y^3,z^3)^{\top})\right\}\subset \R\times D(\mathcal{A}) 
$$
with 
\begin{equation}\label{1pol1}
\abs{\la_n}\to +\infty\quad \text{and}\quad \|U\|_{\mathcal{H}}=\|(u^1,v^1,u^3,v^3,\omega,\psi,y^1,z^1,y^3,z^3)\|_{\mathcal{H}}=1,
\end{equation}
such that 
\begin{equation}\label{2pol2}
\la_n^{\ell}\left(i\la_n I-\mathcal{A}\right)U_n=F_n:=(f_{1,n},f_{2,n},f_{3,n},f_{4,n},f_{5,n},f_{6,n},f_{7,n},f_{8,n},f_{9,n},f_{10,n})^{\top}\to 0\ \ \text{in}\ \ \mathcal{H}. 
\end{equation}
For simplicity, we drop the index $n$. Equivalently, from \eqref{pol2}, we have 
\begin{eqnarray}
i\la u^1-v^1&=&\la^{-\ell}f_1,\label{pol1}\\[0.1in]
i\la v^1-(\rho_1h_1)^{-1}\left[E_1h_1u^1_{xx}+\tau -a v^1\right]&=&\la^{-\ell}f_2,\label{pol2}\\[0.1in]
i\la y^1-z^1&=&\la^{-\ell}f_3,\label{pol7}\\[0.1in]
i\la z^1-(\rho_1I_1)^{-1}\left[E_1I_1y^1_{xx}+\frac{h_1}{2}\tau-G_1h_1(\omega_x+y^1)-bz^1\right]&=&\la^{-\ell}f_4,\label{pol8}\\[0.1in]
i\la \omega-\psi&=&\la^{-\ell}f_5,\label{pol5}\\[0.1in]
i\la \psi-(\rho h)^{-1}\left[-EI\omega_{xxxx}+G_1h_1(\omega_x+y^1)_x+G_3h_3(\omega_x+y^3)_x+h_2\tau_x-c\psi\right]&=&\la^{-\ell}f_6,\label{pol6}\\[0.1in]
i\la u^3-v^3&=&\la^{-\ell}f_7,\label{pol3}\\[0.1in]
i\la v^3-(\rho_3 h_3)^{-1}\left[E_3h_3u^3_{xx}-\tau-dv^3\right]&=&\la^{-\ell}f_8,\label{pol4}\\[0.1in]
i\la y^3-z^3&=&\la^{-\ell}f_9,\label{pol9}\\[0.1in]
i\la z^3-(\rho_3I_3)^{-1}\left[E_3I_3y_{xx}^3+\frac{h_3}{2}\tau-G_3h_3(\omega_x+y^3)-ez^3\right]&=&\la^{-\ell}f_{10},\label{pol10}
\end{eqnarray}
where 
\begin{equation}\label{tau}
\tau=-u^1+u^3+h_2\omega_x-\frac{h_1}{2}y^1-\frac{h_3}{2}y^3.
\end{equation}
\subsection{Proof of Theorem \ref{POL-ab}.}
In this subsection, {assuming that \eqref{A1} holds, we shall find a contradiction with \eqref{1pol1}}. For clarity, we divide the proof into several Lemmas. 
\begin{Lemma}\label{dissipation}
The solution $U\in D(\mathcal{A})$ of system \eqref{pol1}-\eqref{pol10} satisfies the following asymptotic {estimations} 
\begin{equation}\label{L1-est1}
\int_0^L\abs{v^1}^2dx=\frac{o(1)}{\la^{\ell}},\ \ \int_0^L\abs{z^1}^2dx=\frac{o(1)}{\la^{\ell}},\ \ \int_0^L\abs{u^1}^2dx=\frac{o(1)}{\la^{\ell+2}},\ \int_0^L\abs{y^1}^2dx=\frac{o(1)}{\la^{\ell+2}}\ \text{and}\ \|\tau\|=\frac{O(1)}{\la^{\frac{1}{2}}}.
\end{equation}
\end{Lemma}
\begin{proof}
Taking the inner product of $F$ with $U$ in $\mathcal{H}$, then using \eqref{1pol1} and the fact that $U$ is uniformly bounded in $\mathcal{H}$, we get 
\begin{equation}\label{L1-est2}
a\int_0^L\abs{v^1}^2dx+c\int_0^L\abs{z^1}^2dx=-\Re\left(\left<\mathcal{A}U,U\right>_{\mathcal{H}}\right)=\Re\left(\left<(i\la I-\mathcal{A})U,U\right>_{\mathcal{H}}\right)=o\left(\la^{-\ell}\right).
\end{equation}
Then, we obtain the first two estimations in \eqref{L1-est1}. Now, using \eqref{L1-est2}, \eqref{pol1}, \eqref{pol7} and the facts that $f_1, f_3\to 0$ in $H_0^1(0,L)$, we get 
\begin{equation*}
\int_0^L\abs{u^1}^2dx\leq \frac{2}{\la^2}\int_0^L\abs{v^1}^2dx+\frac{2}{\la^{2+2\ell}}\int_0^L\abs{f_1}^2dx=\frac{o(1)}{\la^{\ell+2}}\ \text{and}\ \int_0^L\abs{y^1}^2dx\leq \frac{2}{\la^2}\int_0^L\abs{z^1}^2dx+\frac{2}{\la^{2+2\ell}}\int_0^L\abs{f_3}^2dx=\frac{o(1)}{\la^{\ell+2}}.
\end{equation*}
{ This gives} the third and the fourth estimations of \eqref{L1-est1}. { By} the Gagliardo-Nirenberg interpolation inequality and the facts that $\la \omega$ and $\omega_{xx}$ are uniformly bounded in $L^2(0,L)$, we get 
\begin{equation}\label{L1-est3}
\|\omega_{x}\|_{L^2(0,L)}\leq \|\omega_{xx}\|_{L^2(0,L)}^{\frac{1}{2}}\|\omega\|^{\frac{1}{2}}_{L^2(0,L)}+\|\omega\|_{L^2(0,L)}\leq O(\la^{-\frac{1}{2}}).
\end{equation}
Using the facts that $\la u^1,\la u^3,\la y^1,\la y^3$ are uniformly bounded in $L^2(0,L)$ and \eqref{L1-est3} in \eqref{tau}, we get the last equation in \eqref{L1-est1}. The proof has been completed.  
\end{proof} 
\begin{Lemma}\label{L2-est}
For all $\ell\geq 3$, the solution $U\in D(\mathcal{A})$ of system \eqref{pol1}-\eqref{pol10} satisfies the following asymptotic behavior estimations
\begin{equation}\label{L2-est2}
\int_0^L\abs{u_x^1}^2dx=\frac{o(1)}{\la^{\frac{\ell}{2}+\frac{3}{2}}}\quad \text{and}\quad \int_0^L\abs{y_x^1}^2dx=\frac{o(1)}{\la^{\frac{\ell}{2}+\frac{3}{2}}}.
\end{equation}
\end{Lemma}
\begin{proof}
Multiplying \eqref{pol2} and \eqref{pol8} respectively  by $u_1$ and $y_1$, using integration by parts on $(0,L)$ and the fact that $f_2,f_4\to 0$ in $L^2(0,L)$ and $\la u^1, \la y^1$ are uniformly bounded in $L^2(0,L)$, we get 
\begin{equation}\label{L2-est3}
i\la \int_0^Lv^1\overline{u^1}dx+\frac{E_1}{\rho_1}\int_0^L\abs{u^1_x}^2dx-\frac{1}{h_1\rho_1}\int_0^L\tau \overline{u^1}dx+\frac{a}{\rho_1h_1}\int_0^Lv^1\overline{u^1}dx=\frac{o(1)}{\la^{\ell+1}}.
\end{equation}
and 
\begin{equation}\label{L2-est4}
i\la\int_0^Lz^1\overline{y^1}dx+\frac{E_1}{\rho_1}\int_0^L\abs{y^1_x}^2dx-\frac{h_1}{2\rho_1I_1}\int_0^L\tau \overline{y^1}dx+\frac{G_1 h_1}{\rho_1 I_1}\int_0^L (\omega_x+y^1)\overline{y^1}dx+\frac{b}{\rho_1I_1}\int_0^Lz^1\overline{y^1}dx=\frac{o(1)}{\la^{\ell+1}}.
\end{equation}
Using Lemma \ref{dissipation} and \eqref{L1-est3} in \eqref{L2-est3} and \eqref{L2-est4}, we get 
\begin{equation}\label{L2-est5}
\left\{\begin{array}{lll}
\displaystyle
\left|i\la \int_0^Lv^1\overline{u^1}dx\right|=\frac{o(1)}{\la^{\ell}},&
\displaystyle
\left|\int_0^L\tau \overline{u^1}dx\right|=\frac{o(1)}{\la^{\frac{\ell}{2}+\frac{3}{2}}},&\displaystyle
 \left|\int_0^Lv^1\overline{u^1}dx\right|=\frac{o(1)}{\la^{\ell+1}}\\[0.2in]
\displaystyle
\left|i\la \int_0^Lz^1\overline{y^1}dx\right|=\frac{o(1)}{\la^{\ell}}&\displaystyle
\left|\int_0^L\tau \overline{y^1}dx\right|=\frac{o(1)}{\la^{\frac{\ell}{2}+\frac{3}{2}}},
&
\displaystyle
\left|\int_0^L (\omega_x+y^1)\overline{y^1}dx\right|=\frac{o(1)}{\la^{\frac{\ell}{2}+\frac{3}{2}}},\\[0.2in]
\displaystyle
\left|\int_0^Lz^1\overline{y^1}dx\right|=\frac{o(1)}{\la^{\ell+1}}.&&
\end{array}
\right.
\end{equation}
Inserting \eqref{L2-est5} in \eqref{L2-est3} and \eqref{L2-est4}, we get the desired result. The proof has been completed.
\end{proof}
\begin{Lemma}\label{L3-est}
For $\ell\geq 3$, the solution $U\in D(\mathcal{A})$ of system \eqref{pol1}-\eqref{pol10} satisfies the following asymptotic behavior estimations 
\begin{equation}\label{1L3-est1}
\|\omega_{xxx}\|_{L^2(0,L)}\leq O(\la^{\frac{1}{2}}),\ \ \|\omega\|_{H^4(0,L)}\leq O(\la),\ \abs{\omega_{xxx}(\zeta)}\leq O(\la^{\frac{3}{4}}),\ \abs{u_x(\zeta)}\leq \frac{o(1)}{\la^{\frac{\ell}{8}+\frac{5}{8}}},\ \zeta\in \{0,L\}.
\end{equation}
\end{Lemma}
\begin{proof}
From \eqref{pol6}, using the fact that $\psi, (\omega_x+y^1)_x$, $(\omega_x+y^3)_x$ and $\tau_x$ are uniformly bounded in $L^2(0,L)$ and $f_6\to 0$ in $H_0^2(0,L)$, we obtain
\begin{equation}\label{L3-est1}
\|\omega_{xxxx}\|_{L^2(0,L)}=O\left(\la\right).
\end{equation}
Using Gagliardo-Nirenberg inequality, \eqref{L3-est1} and the fact that $\omega_{xx}$ is uniformly bounded in $L^2(0,L)$, we get 
\begin{equation}\label{L3-est2}
\|\omega_{xxx}\|_{L^2(0,L)}\leq \|\omega_{xxxx}\|^{\frac{1}{2}}\|\omega_{xx}\|^{\frac{1}{2}}+\|\omega_{xx}\|\leq O(\la^{\frac{1}{2}}).
\end{equation}
Using \eqref{L3-est1} and \eqref{L3-est2}, we get the second estimation in \eqref{1L3-est1}.
Since $\omega\in H^4(0,L)$, then $\omega_{xxx}\in H^1(0,L)\subset C\left([0,L]\right)$. Thus, by using Gagliardo-Nirenberg inequality,  we obtain 
\begin{equation}\label{L3-est3}
|\omega_{xxx}(\zeta)|\leq \|\omega_{xxxx}\|_{L^2(0,L)}^{\frac{1}{2}}\|\omega_{xxx}\|^{\frac{1}{2}}_{L^2(0,L)}+\|\omega_{xxx}\|_{L^2(0,L)}\leq O(\la^{\frac{3}{4}}). 
\end{equation}
Since, $u\in H^2(0,L)$, then $u_x\in H^1(0,L)\subset C\left([0,L]\right)$. Thus, by using Gagliardo-Nirenberg inequality, we get 
\begin{equation}\label{L3-est4}
\abs{u^1_x(\zeta)}\leq \|u^1_{xx}\|^{\frac{1}{2}}_{L^2(0,L)}\|u^1_x\|^{\frac{1}{2}}_{L^2(0,L)}+\|u^1_x\|_{L^2(0,L)}.
\end{equation}
Using \eqref{pol2} and \eqref{L1-est1} and the fact that $\ell\geq 3$, we get 
\begin{equation}\label{L3-est5}
\|u^1_{xx}\|\leq O(\la^{-\frac{1}{2}})\quad \text{and}\quad \|u^1_x\|=\frac{o(1)}{\la^{\frac{\ell}{4}+\frac{3}{4}}}.
\end{equation}
Inserting \eqref{L3-est5} in \eqref{L3-est4}, we get the desired results. The proof has been completed.
\end{proof}
\begin{Lemma}\label{L4-est}
For $\ell\geq 3$, the solution $U\in D(\mathcal{A})$ of system \eqref{pol1}-\eqref{pol10} satisfies the following asymptotic behavior estimations 
\begin{equation}\label{L4-est1}
\int_0^L \abs{\omega_{xx}}^2dx=o(1)\quad \text{and}\quad \int_0^L\abs{\psi}^2dx=o(1).
\end{equation}
\end{Lemma}
\begin{proof}
The proof of this lemma is divided into two steps.\\
\textbf{Step 1.} Multiplying \eqref{pol2} by $\overline{\omega}_{xxx}$, integrating over $(0,L)$ and using the fact that $f_2\to 0$ in $L^2(0,L)$ and \eqref{L3-est1}, we get 
\begin{equation}\label{L4-est2}
i\la \int_0^Lv^1\overline{\omega}_{xxx}dx-\frac{E_1}{\rho_1}\int_0^Lu_{xx}^1\overline{\omega}_{xxx}dx-\frac{1}{\rho_1 h_1}\int_0^L\tau \overline{\omega}_{xxx}+\frac{a}{\rho_1h_1}\int_0^Lv^1\overline{\omega}_{xxx}dx=\frac{o(1)}{\la^{\ell-\frac{1}{2}}}.
\end{equation}
Using \eqref{L1-est1} and  \eqref{L3-est1}, we get 
\begin{equation}\label{L4-est3}
\left|i\la \int_0^L v^1\overline{\omega}_{xxx}dx\right|=\frac{o(1)}{\la^{\frac{\ell-3}{2}}}\quad \text{and}\quad \left|i\int_0^L v^1\overline{\omega}_{xxx}dx\right|=\frac{o(1)}{\la^{\frac{\ell-1}{2}}}.
\end{equation}
Now, using integration by parts to the second term in \eqref{L4-est2},  \eqref{L2-est2} and \eqref{1L3-est1}, we obtain 
\begin{equation}\label{L4-est4}
\left|-\frac{E_1}{\rho_1}\int_0^Lu^1_{xx}\overline{\omega}_{xxx}dx\right|=
\left|\frac{E_1}{\rho_1}\int_0^Lu_x^1\overline{\omega}_{xxxx}dx-\frac{E_1}{\rho_1}u^1_x(L)\overline{\omega}_{xxx}(L)+\frac{E_1}{\rho_1}u^1_x(0)\overline{\omega}_{xxx}(L)\right|=\frac{o(1)}{\la^{\frac{\ell}{8}-\frac{1}{8}}}.
\end{equation}
On the other hand, using the definition of $\tau$ in \eqref{tau} and the fact that $\la u^1,\la u^3,\la y^1,\la y^3$ are uniformly bounded in $L^2(0,L)$ and \eqref{1L3-est1}, we get 
\begin{equation}\label{L4-est5}
-\frac{1}{\rho_1h_1}\int_0^L\tau\overline{\omega}_{xxx}dx=\frac{h_2}{\rho_1h_1}\int_0^L\abs{\omega_{xx}}^2dx+\frac{O(1)}{\la^{\frac{1}{2}}}=\frac{h_2}{\rho_1h_1}\int_0^L\abs{\omega_{xx}}^2dx+o(1). 
\end{equation}
Inserting \eqref{L4-est3}-\eqref{L4-est5} in \eqref{L4-est2} and using the fact that $\ell\geq 3$, we get 
the first estimation in \eqref{L4-est1}.\\
\textbf{Step 2.} Multiplying \eqref{pol6} by $-\overline{\omega}$, integrating over $(0,L)$, using the facts that $(\omega_x+y^1)$, $(\omega_x+y^3)$, $\la \omega$ are uniformly bounded in $L^2(0,L)$ and $\|\omega_x\|\leq O(\la^{-\frac{1}{2}})$, \eqref{L1-est1},  the fact that $f_2\to 0$ in $L^2(0,L)$, we get 
\begin{equation}\label{L4-est6}
-i\la \int_0^L\psi \overline{\omega}dx-\frac{EI}{\rho h}\int_0^L\omega_{xxxx}\overline{\omega}dx=o(1).
\end{equation}
Using integration by parts to the second term in \eqref{L4-est6}, and  \eqref{pol5}, we get 
\begin{equation}\label{L4-est7}
\int_0^L\abs{\psi^2}dx-\frac{EI}{\rho h}\int_0^L\abs{\omega_{xx}}^2dx=o(1).
\end{equation}
Finally, using the first estimation in \eqref{L4-est1} in \eqref{L4-est7}, we obtain the desired result. The proof has been completed. 
\end{proof}
\begin{Lemma}\label{L5}
Forall $\ell\geq 3$,  the solution {$U\in D(\mathcal{A})$} of system \eqref{pol1}-\eqref{pol10} satisfies the following asymptotic behavior estimation
\begin{equation}\label{L5-est}
\begin{array}{l}
\displaystyle
\int_0^L\abs{u_x^3}^2dx=
-\la^2\frac{\rho_1h_1\rho_3}{E_3}\left(\frac{E_3}{\rho_3}-\frac{E_1}{\rho_1}\right)\left(\frac{4G_3+h_3}{4G_3}\right)\Re\left(\int_0^Lu_x^3\overline{u^1_x}dx\right)\\[0.1in]
\displaystyle
-\la^2\frac{\rho_1h_1\rho_3I_3}{2G_3h_3E_3}\left(\frac{E_3}{\rho_3}-\frac{E_1}{\rho_1}\right)\Re\left(\int_0^Lu_x^1\overline{y^3_x}dx\right)+o(1)
\end{array}
\end{equation}
and
\begin{equation}\label{L5-est1}
\begin{array}{l}
\displaystyle
\int_0^L\abs{y_x^3}^2dx=2\la^2\frac{\rho_1h_1\rho_3}{E_3h_3}\left(\frac{E_3}{\rho_3}-\frac{E_1}{\rho_1}\right)\left(1-\frac{I_3}{G_3h_3^2}\right)\Re\left(\int_0^Lu_x^1\overline{y^3_x}dx\right)\\[0.1in]
\displaystyle
-\la^2\frac{\rho_1h_1\rho_3}{h_3E_3G_3}\left(\frac{E_3}{\rho_3}-\frac{E_1}{\rho_1}\right)\Re\left(\int_0^Lu_x^3\overline{u^1_x}dx\right)+o(1).
\end{array}
\end{equation}
\end{Lemma}
\begin{proof}
The proof of this Lemma is divided into several steps.\\
\textbf{Step 1.} In this step, we will prove the following estimation 
\begin{equation}\label{L5-step1-1}
\la^2\frac{\rho_1h_1\rho_3}{E_3}\left(\frac{E_3}{\rho_3}-\frac{E_1}{\rho_1}\right)\Re\left(\int_0^Lu_x^3\overline{u^1_x}dx\right)+\int_0^L\abs{u_x^3}^2dx-\frac{h_3}{2}\Re\left(\int_0^Ly_x^3\overline{u_x^3}dx\right)=o(1).
\end{equation}
For this aim, multiplying \eqref{pol2} and \eqref{pol4} by $\frac{E_3}{\rho_3}\overline{u_{xx}^3}$ and $\frac{E_1}{\rho_1}\overline{u_{xx}^1}$ respectively, integrating over $(0,L)$,  using the fact that $\|u^3_{xx}\|=O(\la)$, $\|u^1_{xx}\|=O(\la^{-\frac{1}{2}})$ and $f_2,f_8\to 0$ in $L^2(0,L)$, we get 
\begin{equation}\label{L5-step1-2}
i\la\frac{E_3}{\rho_3}\int_0^Lv^1\overline{u^3_{xx}}dx-\frac{E_1E_3}{\rho_1\rho_3}\int_0^Lu^1_{xx}\overline{u^3_{xx}}dx-\frac{E_3}{\rho_1h_1\rho_3}\int_0^L\tau \overline{u^3_{xx}}dx+\frac{aE_3}{\rho_1h_1\rho_3}\int_0^Lv^1\overline{u^3_{xx}}dx=\frac{o(1)}{\la^{\ell-1}}
\end{equation}
and 
\begin{equation}\label{L5-step1-3}
i\la\frac{E_1}{\rho_1}\int_0^Lv^3\overline{u^1_{xx}}dx-\frac{E_1E_3}{\rho_1\rho_3}\int_0^Lu^3_{xx}\overline{u^1_{xx}}dx+\frac{E_1}{\rho_1h_3\rho_3}\int_0^L\tau \overline{u^1_{xx}}dx=\frac{o(1)}{\la^{\ell+\frac{1}{2}}}.
\end{equation}
Using \eqref{pol1} and \eqref{pol3} in the first term in \eqref{L5-step1-2} and \eqref{L5-step1-3}, using integration by parts and the facts that $f_1,f_7\to 0$ in $H_0^1(0,L)$ and $u^1_x$, $u^3_x$ are uniformly bounded in $L^2(0,L)$, we get 
\begin{equation}\label{L5-step1-4}
i\la\frac{E_3}{\rho_3}\int_0^Lv^1\overline{u^3_{xx}}dx=\la^2\frac{E_3}{\rho_3}\int_0^Lu_x^1\overline{u_x^3}dx+\frac{o(1)}{\la^{\ell-1}}\ \text{and}\ i\la\frac{E_1}{\rho_1}\int_0^Lv^3\overline{u^1_{xx}}dx=\la^2\frac{E_1}{\rho_1}\int_0^Lu_x^3\overline{u_x^1}dx+\frac{o(1)}{\la^{\ell-1}}.
\end{equation}
Now, using \eqref{tau} and integrating by parts over $(0,L)$ and using the fact that $u_x^3$ is uniformly bounded in $L^2(0,L)$, Lemma \ref{L4-est} and Lemma  \ref{L2-est}, we get 
\begin{equation}\label{L5-step1-5}
-\frac{E_3}{\rho_1h_1\rho_3}\int_0^L\tau\overline{u_{xx}^3}dx=\frac{E_3}{\rho_1h_1h_3}\int_0^L\abs{u_x^3}^2dx-\frac{E_3h_3}{2\rho_1h_1\rho_3}\int_0^Ly_x^3\overline{u_x^3}dx+o(1).
\end{equation}
and 
\begin{equation}\label{L5-step1-6}
\frac{E_1}{\rho_1h_3\rho_3}\int_0^L\tau \overline{u^1_{xx}}dx=-\frac{E_1}{\rho_1h_3\rho_3}\int_0^L\tau_x \overline{u^1_{x}}dx=o(1).
\end{equation}
Using integration by parts to the last term in the left hand side of \eqref{L5-step1-2}, and using the fact that $u^3_x$ is uniformly bounded in $L^2(0,L)$ and Lemma \ref{L2-est}, we get 
\begin{equation}\label{L5-step1-7}
\frac{aE_3}{\rho_1h_1\rho_3}\int_0^Lv^1\overline{u^3_{xx}}dx=-\frac{aE_3}{\rho_1h_1\rho_3}\int_0^Lv^1_x\overline{u^3_{x}}dx=\frac{o(1)}{\la^{\frac{\ell}{4}-\frac{1}{4}}}.
\end{equation}
Inserting \eqref{L5-step1-4}, \eqref{L5-step1-5} and \eqref{L5-step1-7} in \eqref{L5-step1-2}, \eqref{L5-step1-4} and \eqref{L5-step1-6} in \eqref{L5-step1-3},  and using the facts that $\ell\geq 3$, we get 
\begin{equation}\label{L5-step1-8}
\begin{array}{c}
\displaystyle
\la^2\frac{E_3}{\rho_3}\int_0^Lu_x^1\overline{u_x^3}dx-\frac{E_1E_3}{\rho_1\rho_3}\int_0^Lu^1_{xx}\overline{u^3_{xx}}dx+\frac{E_3}{\rho_1h_1h_3}\int_0^L\abs{u_x^3}^2dx\\[0.1in]
\displaystyle
-\frac{E_3h_3}{2\rho_1h_1\rho_3}\int_0^Ly_x^3\overline{u_x^3}dx=o(1)
\end{array}
\end{equation}
and 
\begin{equation}\label{L5-step1-9}
\la^2\frac{E_1}{\rho_1}\int_0^Lu_x^3\overline{u_x^1}dx-\frac{E_1E_3}{\rho_1\rho_3}\int_0^Lu^3_{xx}\overline{u^1_{xx}}dx=o(1).
\end{equation}
Subtracting \eqref{L5-step1-8} and \eqref{L5-step1-9} and taking the real part, we get the desired equation \eqref{L5-step1-1}.\\
\textbf{Step 2.}The aim of this step is to prove the following estimation
\begin{equation}\label{L5-step2-1}
\la^2\frac{\rho_1h_1\rho_3}{E_3}\left(\frac{E_3}{\rho_3}-\frac{E_1}{\rho_1}\right)\Re\left(\int_0^Lu_x^1\overline{y^3_x}dx\right)+\Re\left(\int_0^Lu_x^3\overline{y_x^3}dx\right)-\frac{h_3}{2}\int_0^L\abs{y_x^3}^2dx=o(1).
\end{equation}
For this aim, multiplying \eqref{pol2} and \eqref{pol10}  by $\frac{1}{2}h_3\frac{E_3}{\rho_3}\overline{y^3_{xx}}$ and $\frac{1}{2}h_3\frac{E_1}{\rho_1}\overline{u^1_{xx}}$ respectively, integrating over $(0,L)$, and using the fact that $\|y^3_{xx}\|=O(\la)$ and $\|u^1_{xx}\|=O(\la^{-\frac{1}{2}})$ and $f_2,f_{10}\to 0$ in $L^2(0,L)$, we get 
\begin{equation}\label{L5-step2-2}
i\la h_3\frac{E_3}{2\rho_3}\int_0^Lv^1\overline{y^3_{xx}}dx-\frac{E_1E_3h_3}{2\rho_1\rho_3}\int_0^Lu^1_{xx}\overline{y^3_{xx}}dx-\frac{E_3h_3}{2\rho_1h_1\rho_3}\int_0^L\tau \overline{y^3_{xx}}dx+\frac{aE_3h_3}{2\rho_1h_1\rho_3}\int_0^Lv^1\overline{y^3_{xx}}dx=\frac{o(1)}{\la^{\ell-1}}
\end{equation}
and 
\begin{equation}\label{L5-step2-3}
\begin{array}{c}
\displaystyle
i\la h_3\frac{E_1}{2\rho_1}\int_0^Lz^3\overline{u^1_{xx}}dx-\frac{E_1E_3h_3}{2\rho_1\rho_3}\int_0^Ly^3_{xx}\overline{u^1_{xx}}dx-\frac{h_3^2E_1}{4\rho_1\rho_3I_3}\int_0^L\tau\overline{u^1_{xx}}dx\\[0.1in]
\displaystyle
+\frac{G_3h_3^2E_1}{2\rho_1\rho_3I_3}\int_0^L(\omega_x+y^3)\overline{u^1_{xx}}dx=\frac{o(1)}{\la^{\ell+\frac{1}{2}}}.
\end{array}
\end{equation}
Using \eqref{pol1} and \eqref{pol9} in the first term in \eqref{L5-step2-2} and \eqref{L5-step2-3}, using integration by parts and the facts that $f_1,f_9\to 0$ in $H_0^1(0,L)$ and $y^3_x$ and $u^1_x$ are uniformly bounded in $L^2(0,L)$, we get 
\begin{equation}\label{L5-step2-4}
i\la\frac{E_3h_3}{2\rho_3}\int_0^Lv^1\overline{y^3_{xx}}dx=\la^2\frac{E_3h_3}{2\rho_3}\int_0^Lu_x^1\overline{y_x^3}dx+\frac{o(1)}{\la^{\ell-1}},
\end{equation}
and
\begin{equation}\label{L5-step2-5}
i\la\frac{E_1h_3}{2\rho_1}\int_0^Lz^3\overline{u^1_{xx}}dx=\la^2\frac{E_1h_3}{2\rho_1}\int_0^Ly_x^3\overline{u_x^1}dx+\frac{o(1)}{\la^{\ell-1}}.
\end{equation}
Now, using integration by parts over $(0,L)$,  \eqref{tau} and  the facts that $y_x^3$ and $\tau_x$ are uniformly bounded in $L^2(0,L)$, Lemma \ref{L4-est} and  \ref{L2-est}, we get 
\begin{equation}\label{L5-step2-6}
-\frac{E_3h_3}{2\rho_1h_1\rho_3}\int_0^L\tau \overline{y^3_{xx}}dx=-\frac{E_3h_3^2}{4\rho_1h_1\rho_3}\int_0^L\abs{y_x^3}^2dx+\frac{E_3h_3}{2\rho_1h_1\rho_3}\int_0^Lu_x^3\overline{y_x^3}dx+o(1).
\end{equation}
and 
\begin{equation}\label{L5-step2-7}
-\frac{h_3^2E_1}{4\rho_1\rho_3I_3}\int_0^L\tau\overline{u^1_{xx}}dx=\frac{h_3^2E_1}{4\rho_1\rho_3I_3}\int_0^L\tau_x\overline{u^1_{x}}dx=\frac{o(1)}{\la^{\frac{\ell+3}{4}}}.
\end{equation}
Using integration by parts to the last term in left hand side of \eqref{L5-step2-2} and using the fact that $y_x^3$ is uniformly bounded in $L^2(0,L)$ and Lemma \ref{L2-est}, we get 
\begin{equation}\label{L5-step2-8}
\frac{aE_3h_3}{2\rho_1h_1\rho_3}\int_0^Lv^1\overline{y^3_{xx}}dx=-\frac{aE_3h_3}{2\rho_1h_1\rho_3}\int_0^Lv^1_x\overline{y^3_{x}}dx=\frac{o(1)}{\la^{\frac{\ell-1}{4}}}.
\end{equation}
Using integration by parts to the last term in the left hand side of \eqref{L5-step2-3}, and using the fact that $u_x^1, y_x^3$ are uniformly bounded in $L^2(0,L)$ and Lemma  \ref{L4-est}, we get 
\begin{equation}\label{L5-step2-9}
\frac{G_3h_3^2E_1}{2\rho_1\rho_3I_3}\int_0^L(\omega_x+y^3)\overline{u^1_{xx}}dx=o(1). 
\end{equation}
Inserting \eqref{L5-step2-4}, \eqref{L5-step2-6} and \eqref{L5-step2-8} in \eqref{L5-step2-2}, \eqref{L5-step2-5}, \eqref{L5-step2-7} and \eqref{L5-step2-9} in \eqref{L5-step2-3} and the fact that $\ell\geq 3$, we get 
\begin{equation}\label{L5-step2-10}
\la^2\frac{E_3h_3}{2\rho_3}\int_0^Lu_x^1\overline{y_x^3}dx-\frac{E_1E_3h_3}{2\rho_1\rho_3}\int_0^Lu^1_{xx}\overline{y^3_{xx}}dx-\frac{E_3h_3^2}{4\rho_1h_1\rho_3}\int_0^L\abs{y_x^3}^2dx+\frac{E_3h_3}{2\rho_1h_1\rho_3}\int_0^Lu_x^3\overline{y_x^3}dx=o(1)
\end{equation}
and
\begin{equation}\label{L5-step2-11}
\la^2\frac{E_1h_3}{2\rho_1}\int_0^Ly_x^3\overline{u_x^1}dx-\frac{E_1E_3h_3}{2\rho_1\rho_3}\int_0^Ly^3_{xx}\overline{u^1_{xx}}dx=o(1).
\end{equation}
Subtracting \eqref{L5-step2-10} and \eqref{L5-step2-11} and taking the real part, we reach \eqref{L5-step2-1}.\\
\textbf{Step 3.} The aim of this step is to prove the following estimation 
\begin{equation}\label{L5-step3-1}
\frac{1}{2}\int_0^L\abs{y_x^3}^2dx-\frac{h_3}{2I_3}\int_0^L\abs{u_x^3}^2dx+\left(-\frac{1}{h_3}+\frac{h_3^2}{4I_3}+\frac{G_3h_3}{I_3}\right)\Re\left(\int_0^Ly_x^3\overline{u^3_x}dx\right)=o(1).
\end{equation}
For this aim, multiplying \eqref{pol4} and \eqref{pol10}  by $\overline{y^3_{xx}}$ and $\overline{u_{xx}^3}$ respectively, integrating over $(0,L)$,  using  \eqref{pol3} and \eqref{pol9}, the facts that $f_7,f_9\to 0$ in $H_0^1(0,L)$ and $\|u_{xx}^3\|,\|y^3_{xx}\|=O(\la)$, and taking the real part of their subtract, we get 
\begin{equation}\label{L5-step3-2}
\frac{1}{\rho_3h_3}\Re\left(\int_0^L\tau\overline{y^3_{xx}}dx\right)+\frac{h_3}{2\rho_3I_3}\Re\left(\int_0^L\tau\overline{u_{xx}^3}dx\right)-\frac{G_3h_3}{\rho_3I_3}\Re\left(\int_0^L\left(\omega_x+y^3\right)\overline{u^3_{xx}}dx\right)=\frac{o(1)}{\la^{\ell-2}}.
\end{equation}
Using integration by parts, Lemmas \ref{L2-est}, \ref{L4-est}, and the fact that $\overline{u^3_x}$ and $\overline{y^3_x}$ are uniformly bounded in $L^2(0,L)$, we get 
\begin{equation}\label{L5-step3-3}
\left\{\begin{array}{l}
\displaystyle 
\frac{1}{\rho_3h_3}\Re\left(\int_0^L\tau\overline{y^3_{xx}}dx\right)=-\frac{1}{\rho_3h_3}\Re\left(\int_0^Lu^3_x\overline{y^3_x}dx\right)+\frac{1}{2\rho_3}\int_0^L\abs{y^3_x}^2dx+o(1),\\[0.1in]
\displaystyle
\frac{h_3}{2\rho_3I_3}\Re\left(\int_0^L\tau\overline{u_{xx}^3}dx\right)=-\frac{h_3}{2\rho_3I_3}\int_0^L\abs{u^3_x}^2dx+\frac{h_3^2}{4\rho_3I_3}\Re\left(\int_0^Ly_x^3\overline{u_x^3}dx\right)+o(1),\\[0.1in]
\displaystyle
-\frac{G_3h_3}{\rho_3I_3}\Re\left(\int_0^L\left(\omega_x+y^3\right)\overline{u^3_{xx}}dx\right)=\frac{G_3h_3}{\rho_3I_3}\Re\left(\int_0^Ly^3_x\overline{u^3_x}dx\right)+o(1).
\end{array}
\right.
\end{equation}
Inserting \eqref{L5-step3-3} in \eqref{L5-step3-2}, we get the desired result \eqref{L5-step3-1}.\\
\textbf{Step 4.} The aim of this step is to prove the following estimations. 
\begin{equation}\label{L5-step4-1}
\begin{array}{l}
\displaystyle
\Re\left(\int_0^Ly_x^3\overline{u^3_x}dx\right)=-\la^2\frac{\rho_1h_1\rho_3}{2E_3G_3}\left(\frac{E_3}{\rho_3}-\frac{E_1}{\rho_1}\right)\Re\left(\int_0^Lu_x^3\overline{u^1_x}dx\right)\\[0.1in]
\displaystyle
-\la^2\frac{\rho_1h_1\rho_3I_3}{G_3h_3^2E_3}\left(\frac{E_3}{\rho_3}-\frac{E_1}{\rho_1}\right)\Re\left(\int_0^Lu_x^1\overline{y^3_x}dx\right)+o(1).
\end{array}
\end{equation}
For this aim, multiplying \eqref{L5-step1-1} by $\frac{h_3}{2I_3}$, we get 
\begin{equation*}
\la^2\frac{\rho_1h_1\rho_3h_3}{2I_3E_3}\left(\frac{E_3}{\rho_3}-\frac{E_1}{\rho_1}\right)\Re\left(\int_0^Lu_x^3\overline{u^1_x}dx\right)+\frac{h_3}{2I_3}\int_0^L\abs{u_x^3}^2dx-\frac{h_3^2}{4I_3}\Re\left(\int_0^Ly_x^3\overline{u_x^3}dx\right)=o(1).
\end{equation*}
Adding \eqref{L5-step3-1} and the above equation, we get 
\begin{equation}\label{L5-step4-2}
\frac{1}{2}\int_0^L\abs{y_x^3}^2dx+\left(-\frac{1}{h_3}+\frac{G_3h_3}{I_3}\right)\Re\left(\int_0^Ly_x^3\overline{u^3_x}dx\right)+\la^2\frac{\rho_1h_1\rho_3h_3}{2I_3E_3}\left(\frac{E_3}{\rho_3}-\frac{E_1}{\rho_1}\right)\Re\left(\int_0^Lu_x^3\overline{u^1_x}dx\right)=o(1).
\end{equation}
Now, multiplying \eqref{L5-step2-1} by $h_3^{-1}$, we get 
\begin{equation*}
\la^2\frac{\rho_1h_1\rho_3}{h_3E_3}\left(\frac{E_3}{\rho_3}-\frac{E_1}{\rho_1}\right)\Re\left(\int_0^Lu_x^1\overline{y^3_x}dx\right)+\frac{1}{h_3}\Re\left(\int_0^Lu_x^3\overline{y_x^3}dx\right)-\frac{1}{2}\int_0^L\abs{y_x^3}^2dx=o(1).
\end{equation*}
Then, adding the above equation and \eqref{L5-step4-2}, we get \eqref{L5-step4-1}.\\
\textbf{Step 5.} The aim of this step is to proof \eqref{L5-est} and \eqref{L5-est1}.
Inserting \eqref{L5-step4-1} in \eqref{L5-step1-1}, we get \eqref{L5-est}. Similarly, inserting \eqref{L5-step4-1} in \eqref{L5-step2-1}, we get \eqref{L5-est1}. The proof has been completed. 
\end{proof}
\begin{Lemma}\label{L6}
Assume that \eqref{A1} holds. For $\ell\geq 3$, the solution $U\in D(\mathcal{A})$ of system \eqref{pol1}-\eqref{pol10} satisfies the following asymptotic behavior estimations 
\begin{equation}\label{L6-est}
\int_0^L\abs{v^3}^2-\frac{E_3}{\rho_3}\int_0^L\abs{u_x^3}^2dx=\frac{O(1)}{\la^{\frac{3}{2}}}\quad \text{and}\quad \int_0^L\abs{z^3}^2dx-\frac{E_3}{\rho_3}\int_0^L\abs{y_x^3}^2dx=\frac{O(1)}{\la^{\frac{3}{2}}}.
\end{equation}
\end{Lemma}
\begin{proof}
Multiplying \eqref{pol4} and \eqref{pol10}  by $-\overline{u^3}$ and $-\overline{y^3}$ respectively , integrating by parts over $(0,L)$ and  the facts that $\la u^3,\ \la y^3$ are  uniformly bounded in $L^2(0,L)$, Lemma \ref{L1-est1} and $f_8,\ f_{10}\to 0$ in $L^2(0,L)$, we get 
\begin{equation}\label{L6-est-1}
-i\la\int_0^Lv^3\overline{u^3}dx-\frac{E_3}{\rho_3}\int_0^L\abs{u^3_x}^2dx=\frac{O(1)}{\la^{\frac{3}{2}}}\quad \text{and}\quad -i\la\int_0^Lz^3\overline{y^3}dx-\frac{E_3}{\rho_3}\int_0^L\abs{y^3_x}^2dx=\frac{O(1)}{\la^{\frac{3}{2}}}.
\end{equation}
From \eqref{pol3}, we have 
$$
-i\la \overline{u^3}=\overline{v^3}+\frac{\overline{f_7}}{\la^\ell} \quad \text{and}\quad -i\la \overline{y^3}=\overline{z^3}+\frac{\overline{f_9}}{\la^\ell}.
$$
Inserting the above equation in \eqref{L6-est-1} and using the facts that $v^3,\ z^3$ are uniformly bounded in $L^2(0,L)$, $f_7, f_9\to 0$ in $H_0^1(0,L)$, $\ell\geq 3$, we get  \eqref{L6-est}. The proof has been completed
\end{proof}
$\newline$\\
{We are now ready to finish the proof of Theorme \ref{POL-ab}}. We distinguish two cases.\\
\textbf{Case 1.} If $\dfrac{E_1}{\rho_1}=\dfrac{E_3}{\rho_3}$ taking $\ell=3$, then from Lemmas \ref{dissipation}-\ref{L6}, we get 
$$
\begin{array}{lllll}
\displaystyle
\int_0^L\abs{v^1}^2dx=\frac{o(1)}{\la^3},&\displaystyle 
\int_0^L\abs{z^1}^2dx=\frac{o(1)}{\la^3},&\displaystyle
\int_0^L\abs{u_x^1}^2dx=\frac{o(1)}{\la^3},&\displaystyle \int_0^L\abs{u_x^3}^2dx=\frac{o(1)}{\la^3},&\displaystyle \int_0^L\abs{\psi}^2dx=o(1),\\[0.1in]
\displaystyle\int_0^L\abs{\omega_{xx}}^2dx=o(1),&\displaystyle 
\int_0^L\abs{u^3_x}^2dx=o(1),&\displaystyle  \int_0^L\abs{y^3_x}^2dx=o(1),&\displaystyle \int_0^L\abs{v^3}^2dx=o(1),&\displaystyle\int_0^L\abs{z^3}^2dx=o(1).
\end{array}
$$
It follows that {$\|U\|_{\mathcal{H}}=o(1)$}, which contradicts ${\rm (H_2)}$. This implies that 
$$
\sup_{\la\in \mathbb{R}}\|(i\la I-\mathcal{A})^{-1}\|_{\mathcal{H}}{\le} O(\la^{3}).
$$
\textbf{Case 2.} If $\dfrac{E_1}{\rho_1}\neq \dfrac{E_3}{\rho_3}$ and $\ell=5$, then from Lemmas \ref{dissipation}-\ref{L4-est}, we get 
\begin{equation}\label{F1}
\begin{array}{lll}
\displaystyle
\int_0^L\abs{v^1}^2dx=\frac{o(1)}{\la^5},&\displaystyle 
\int_0^L\abs{z^1}^2dx=\frac{o(1)}{\la^5},&\displaystyle
\int_0^L\abs{u_x^1}^2dx=\frac{o(1)}{\la^5},\\[0.1in]
\displaystyle \int_0^L\abs{u_x^3}^2dx=\frac{o(1)}{\la^5},&\displaystyle \int_0^L\abs{\psi}^2dx=o(1),&
\displaystyle\int_0^L\abs{\omega_{xx}}^2dx=o(1),
\end{array}
\end{equation}
Using the fact that $\dfrac{E_1}{\rho_1}\neq \dfrac{E_3}{\rho_3}$ , Lemma \ref{L5} and \eqref{F1}, we get 
\begin{equation}\label{F2}
\int_0^L\abs{u^3_x}^2dx=o(1),\quad \text{and}\quad \int_0^L\abs{y^3_x}^2dx=o(1).
\end{equation}
Finally, using \eqref{F2} and Lemma \ref{L6}, we get 
\begin{equation}\label{F3}
\int_0^L\abs{v^3}^2dx=o(1)\quad \text{and}\quad \int_0^L\abs{z^3}^2dx=o(1).
\end{equation}
Then from \eqref{F1}-\eqref{F3}, we obtain { $\|U\|_{\mathcal{H}}=o(1)$}, which contradicts ${\rm (H_2)}$. This implies that 
$$
\sup_{\la\in \mathbb{R}}\|(i\la I-\mathcal{A})^{-1}\|_{\mathcal{H}}{\le}O(\la^{5}).
$$
The proof has been completed. 
\subsection{Proof of Theorem \ref{POL-ac}.}
In this subsection, we assume that \eqref{A2} holds and  we set the following hypotheses:
\begin{equation}\label{H3}\tag{${\rm H_3}$}
\left(\frac{E_1}{\rho_1}=\frac{E_3}{\rho_3},\quad \frac{G_3h_3}{\rho_3I_3}\neq \frac{G_1h_1}{\rho_1I_1}\right)\quad \text{and}\quad \ell=2,
\end{equation}
\begin{equation}\label{H4}\tag{${\rm H_4}$}
\quad \left(\frac{E_1}{\rho_1}=\frac{E_3}{\rho_3}\quad\text{and}\quad \frac{G_3h_3}{\rho_3I_3}= \frac{G_1h_1}{\rho_1I_1}\right)\quad \text{and}\quad \ell=6,
\end{equation}
\begin{equation}\label{H5}\tag{${\rm H_5}$}
\frac{E_1}{\rho_1}\neq \frac{E_3}{\rho_3}\ \quad \text{and}\quad \ell=6.
\end{equation}
 We will check the condition ${\rm (H_2)}$ by finding a contradiction with \eqref{1pol1}. From \eqref{pol1}-\eqref{pol10}, we obtain the following systems 
\begin{eqnarray}
\la^2u^1+(\rho_1h_1)^{-1}\left[E_1h_1u^1_{xx}+\tau-i\la au^1\right]&=&-\la^{-\ell}f_2-\la^{-\ell}(\frac{a}{\rho_1h_1}+i\la)f_1,\label{new-pol1}\\[0.1in]
\la^2y^1+(\rho_1I_1)^{-1}\left[E_1I_1y^1_{xx}+\frac{h_1}{2}\tau-G_1h_1(\omega_x+y^1)\right]&=&-\la^{-\ell}f_4-i\la^{-\ell+1}f_3,\label{new-pol4}\\[0.1in]
\hspace{0.75cm}\la^2\omega+(\rho h)^{-1}\left[-EI\omega_{xxxx}+G_h\omega_{xx}+G_1h_1y^1_x+G_3h_3y^3_x+h_2\tau_x-i\la c\omega\right]&=&-\la^{-\ell}f_6-\la^{-\ell}(c+i\la)f_5,\label{new-pol3}\\[0.1in]
\la^2u^3+(\rho_3 h_3)^{-1}\left[E_3h_3u^3_{xx}-\tau\right]&=&-\la^{-\ell}f_8-i\la^{-\ell+1}f_7,\label{new-pol2}\\[0.1in]
\la^2y^3+(\rho_3I_3)^{-1}\left[E_3I_3y_{xx}^3+\frac{h_3}{2}\tau-G_3h_3(\omega_x+y^3)\right]&=&-\la^{-\ell}f_{10}-i\la^{-\ell+1}f_9,\label{new-pol5}
\end{eqnarray}
where $G_h=G_1h_1+G_3h_3$. For clarity, we divide the proof into several Lemmas.
\begin{Lemma}\label{2C2-diss}
Assume that \eqref{A2} holds and (\eqref{H3}, or \eqref{H4} or \eqref{H5} holds). The solution $U\in D(\mathcal{A})$ of system \eqref{new-pol1}-\eqref{new-pol5} satisfies the following asymptotic behavior estimations 
\begin{equation}\label{2Case2-diss}
\int_0^L\abs{\psi}^2dx,\ \int_0^L\abs{v^1}^2dx=\frac{o(1)}{\la^{\ell}}\quad \text{and}\quad \int_0^L\abs{\omega}^2dx,\int_0^L\abs{u^1}^2dx=\frac{o(1)}{\la^{\ell+2}}.
\end{equation}
\end{Lemma}
\begin{proof}
Taking the inner product of $F$ with $U$ in $\mathcal{H}$, then using \eqref{1pol1} and the fact that $U$ is uniformly bounded in $\mathcal{H}$, we get 
\begin{equation}\label{C2-L1-est2}
a\int_0^L\abs{v^1}^2dx+c\int_0^L\abs{\psi}^2dx=-\Re\left(\left<\mathcal{A}U,U\right>_{\mathcal{H}}\right)=\Re\left(\left<(i\la I-\mathcal{A})U,U\right>_{\mathcal{H}}\right)=o\left(\la^{-\ell}\right).
\end{equation}
Then, we obtain the first two estimations in \eqref{2Case2-diss}. Now, using \eqref{pol1} ,\eqref{pol5} and the fact that $f_1,f_5\to 0$ in $H^1_0(0,L)$ and $H^2_0(0,L)$, respectively, we get the third and the fourth estimations of \eqref{2Case2-diss}. The proof has been completed. 
\end{proof}
\begin{Lemma}\label{C2-L2-est}
For all $\ell\geq2$, the solution $U\in D(\mathcal{A})$  of system \eqref{new-pol1}-\eqref{new-pol5} satisfies the following asymptotic behavior estimations
\begin{equation}\label{2C2-L2-est1}
\int_0^L\abs{\omega_x}^2dx,\ \int_0^L\abs{\omega_{xx}}^2dx=\frac{o(1)}{\la^{\min({\ell},\frac{\ell}{2}+1)}}=\left\{\begin{array}{lll}\displaystyle{\frac{o(1)}{\la^{2}}}&\text{if}&\eqref{H3}\ \text{holds}\\[0.1in] \displaystyle{\frac{o(1)}{\la^{4}}}&\text{if}&(\eqref{H4}\ \text{or}\ \eqref{H5})\ \text{holds}.\end{array} \right.
\end{equation}
and 
\begin{equation}\label{2C2-L2-est1-2}
\int_0^L\abs{u^1_x}^2dx=\frac{o(1)}{\la^{\min\left(\frac{\ell}{2},\frac{\ell}{4}+1\right)}}=\left\{\begin{array}{lll}\displaystyle{\frac{o(1)}{\la^{2}}}&\text{if}&\eqref{H3}\ \text{holds}\\[0.1in] \displaystyle{\frac{o(1)}{\la^{4}}}&\text{if}&(\eqref{H4}\ \text{or}\ \eqref{H5})\ \text{holds}.\end{array} \right.
\end{equation}
\end{Lemma}
\begin{proof}
First, we will prove the  estimations \eqref{2C2-L2-est1}. For this aim, multiplying \eqref{new-pol3} by $-\overline{\omega}$, integrating by parts on $(0,L)$ and using \eqref{2Case2-diss} and the fact that $y^1_x,y^3_x$ and $\tau_x$ are uniformly bounded in $L^2(0,L)$, and $f_5, f_6\to 0$ in $H^2_0(0,L)$ and $L^2(0,L)$ respectively and $\ell\geq 2$, we get 
\begin{equation*}
\frac{EI}{\rho h}\int_0^L\abs{\omega_{xx}}^2dx+\frac{G_h}{\rho h}\int_0^L\abs{\omega_x}^2dx=\frac{o(1)}{\la^{\min({\ell},\frac{\ell}{2}+1)}}.
\end{equation*}
The above equation yields  the   equation \eqref{2C2-L2-est1} in the both cases \eqref{H3} and \eqref{H4}. Next, we will give the estimation \eqref{2C2-L2-est1-2}. For this aim, multiplying \eqref{new-pol1} by $-\overline{u^1}$, integrating by parts on $(0,L)$ and using \eqref{2Case2-diss} and the fact that $f_1,f_2\to 0$ respectively in $H_0^1(0,L)$ and $L^2(0,L)$, we get
\begin{equation}\label{C2-L2-est2}
\frac{E_1}{\rho_1}\int_0^L\abs{u^1_x}^2dx-\int_0^L\tau \overline{u^1}dx=\frac{o(1)}{\la^{\ell}}. 
\end{equation}
Using the fact that $\la u^3, \la y^1$, $\la y^3$ are uniformly bounded in $L^2(0,L)$,  and the definition of \eqref{tau}, we get 
\begin{equation*}
\left|\int_0^L\tau u^1dx\right|=\frac{o(1)}{\la^{\frac{\ell}{2}+2}}.
\end{equation*}
Inserting the above equation in \eqref{C2-L2-est2} and using \eqref{H3}, \eqref{H4}, \eqref{H5}, we get the desired estimation in \eqref{2C2-L2-est1-2}. The proof has been completed. 
\end{proof}
\begin{Lemma}\label{C2-L3-est}
The solution $U\in D(\mathcal{A})$  of system \eqref{new-pol1}-\eqref{new-pol5} satisfies the following asymptotic behavior estimations 
\begin{equation}\label{C2-L3-est1}
\left\{\begin{array}{l}
\displaystyle
\la^2\frac{\rho_1h_1\rho_3}{E_3}\left(\frac{E_1}{\rho_1}-\frac{E_3}{\rho_3}\right)\Re\left(\int_0^Lu_x^3\overline{u^1_x}dx\right)-\int_0^L\abs{u_x^3}^2dx\\[0.1in]
\displaystyle
\frac{h_1}{2}\Re\left(\int_0^Ly^1_x\overline{u^3_x}dx\right)+\frac{h_3}{2}\Re\left(\int_0^Ly_x^3\overline{u_x^3}dx\right)=\frac{o(1)}{\la^{\min\left(\ell-2,\frac{\ell}{2}-1\right)}},
\end{array}
\right.
\end{equation}
\begin{equation}\label{C2-L3-est2}
\left\{\begin{array}{l}
\displaystyle
\la^2\frac{\rho_1h_1\rho_3}{E_3}\left(\frac{E_1}{\rho_1}-\frac{E_3}{\rho_3}\right)\Re\left(\int_0^Lu_x^1\overline{y^3_x}dx\right)+\frac{h_3}{2}\int_0^L\abs{y_x^3}^2dx\\[0.1in]
\displaystyle
-\Re\left(\int_0^Lu_x^3\overline{y_x^3}dx\right)+\frac{h_1}{2}\Re\left(\int_0^Ly^1_x\overline{y^3_x}dx\right)=\frac{o(1)}{\la^{\min\left(\ell-2,\frac{\ell}{2}-1\right)}},
\end{array}
\right.
\end{equation}
\begin{equation}\label{C2-L3-est3}
\left\{\begin{array}{l}
\displaystyle
-\frac{1}{2}\int_0^L\abs{y_x^3}^2dx+\frac{h_3}{2I_3}\int_0^L\abs{u_x^3}^2dx-\frac{h_3h_1}{4I_3}\Re\left(\int_0^Ly^1_x\overline{u^3_x}dx\right)\\[0.1in]
\displaystyle
-\frac{h_1}{2h_3}\Re\left(\int_0^Ly^1_x\overline{y^3_x}dx\right)+\left(\frac{1}{h_3}-\frac{h_3^2}{4I_3}-\frac{G_3h_3}{I_3}\right)\Re\left(\int_0^Ly_x^3\overline{u^3_x}dx\right)=\frac{o(1)}{\la^{\min\left(\ell-2,\frac{\ell}{2}-1\right)}},
\end{array}
\right.
\end{equation}
\begin{equation}\label{u1y1}
\frac{h_1}{2}\int_0^L\abs{y^1_x}^2dx+\frac{h_3}{2}\Re\left(\int_0^Ly^3_x\overline{y^1_x}dx\right)-\Re\left(\int_0^Lu^3_x\overline{y^1_x}dx\right)=\frac{o(1)}{\la^{\frac{\ell}{2}-1}}.
\end{equation}
Moreover, if $\frac{E_1}{\rho_1}=\frac{E_3}{\rho_3}$, we have 
\begin{equation}\label{u3y1}
\left\{\begin{array}{l}
\displaystyle
\frac{h_1}{2\rho_1I_1}\int_0^L\abs{u_x^3}^2dx-\frac{h_1}{2\rho_3h_3}\int_0^L\abs{y^1_x}^2dx-\frac{1}{2\rho_3}\Re\left(\int_0^Ly^3_x\overline{y^1_x}dx\right)\\[0.1in]
\displaystyle
-\frac{h_1h_3}{4\rho_1I_1}\left(\int_0^Ly^3_x\overline{u^3_x}dx\right)+\left(\frac{1}{\rho_3h_3}-\frac{h_1^2}{4\rho_1I_1}-\frac{G_1h_1}{\rho_1I_1}\right)\Re\left(\int_0^Ly^1_x\overline{u^3_x}dx\right)=\frac{o(1)}{\la^{\min{(\ell-2,\frac{\ell}{4}+\frac{1}{2})}}}.
\end{array}
\right.
\end{equation}
\begin{equation}\label{y1y3}
\left\{\begin{array}{l}
\displaystyle
\frac{h_1h_3}{4I_1\rho_1}\int_0^L\abs{y^3_x}^2dx-\frac{h_3h_1}{4\rho_3I_3}\int_0^L\abs{y^1_x}^2dx-\frac{h_1}{2\rho_1I_1}\Re\left(\int_0^Lu^3_x\overline{y^3_x}dx\right)+\frac{h_3}{2\rho_3I_3}\Re\left(\int_0^Lu^3_x\overline{y^1_x}dx\right)\\[0.1in]
\displaystyle
\left(\frac{h_1^2}{4\rho_1I_1}+\frac{G_1h_1}{\rho_1I_1}-\left(\frac{h_3^2}{4\rho_3I_3}+\frac{G_3h_3}{\rho_3I_3}\right)\right)\Re\left(\int_0^Ly^1_x\overline{y^3_x}dx\right)=\frac{o(1)}{\la^{\min{(\ell-2,\frac{\ell}{4}+\frac{1}{2})}}}.
\end{array}
\right.
\end{equation}
\end{Lemma}
\begin{proof} The proof of this Lemma is divided into four steps.\\ 
\textbf{Step 1.} First, multiplying \eqref{new-pol1} and \eqref{new-pol2}  by $\frac{E_3}{\rho_3}\overline{u^3_{xx}}$ and $\frac{E_1}{\rho_1}\overline{u^1_{xx}}$ respectively,  \eqref{new-pol1} and \eqref{new-pol5} by $\frac{E_3}{\rho_3}\overline{y^3_{xx}}$ and $\frac{E_1}{\rho_1}\overline{u^1_{xx}}$  respectively, \eqref{new-pol2} and \eqref{new-pol5}  $\overline{y^3_{xx}}$ and  $\overline{u^3_{xx}}$ respectively, and then using the same arguments in Step 1, Step 2 and Step 3 in Lemma \ref{L5}, we obtain \eqref{C2-L3-est1}-\eqref{C2-L3-est3}.\\
\textbf{Step 2.} In this step, we prove estimation \eqref{u1y1}. For this aim, multiplying \eqref{new-pol1} and \eqref{new-pol4} respectively  by $\overline{y^1_{xx}}$ and $\overline{u^1_{xx}}$, integrating over $(0,L)$ and using Lemmas \ref{2C2-diss}, \ref{C2-L2-est}, and the fact thats $\la u^1_{xx}$ and $\la y^1_{xx}$ are uniformly bounded in $L^2(0,L)$ and $f_1,f_3\to 0$ in $H_0^1(0,L)$ and $f_2,f_8\to 0$ in $L^2(0,L)$, we get 
\begin{equation}\label{step2-eq1}
-\la^2\int_0^Lu^1_x\overline{y^1_x}dx+\frac{E_1}{\rho_1}\int_0^Lu^1_{xx}\overline{y^1_{xx}}dx-\frac{1}{\rho_1h_1}\int_0^L\tau_x \overline{y^1_x}dx=\frac{o(1)}{\la^{\frac{\ell}{2}-1}}.
\end{equation}
\begin{equation}\label{step2-eq2}
-\la^2\int_0^Ly^1_x\overline{u^1_x}dx+\frac{E_1}{\rho_1}\int_0^Ly^1_{xx}\overline{u^1_{xx}}dx-\frac{h_1}{2\rho_1I_1}\int_0^L\tau_x\overline{u^1_{x}}dx+\frac{G_1h_1}{\rho_1I_1}\int_0^L(\omega_x+y^1)\overline{u^1_{x}}dx=\frac{o(1)}{\la^{\ell-2}}.
\end{equation}
Using the fact that $\tau_x$ is uniformly bounded in $L^2(0,L)$, $\la y^1$ is uniformly bounded in $L^2(0,L)$ and Lemma \ref{2C2-diss}, \ref{C2-L2-est}, we get 
\begin{equation}\label{step2-eq3}
\left|\int_0^L\tau_x\overline{u^1_{x}}dx\right|=\frac{o(1)}{\la^{\frac{\ell}{2}}}\quad \text{and}\quad \left|\int_0^L(\omega_x+y^1)\overline{u^1_x}dx\right|={\frac{o(1)}{\la^{\min\left(\frac{\ell}{2},\frac{\ell}{4}+1\right)}}}.
\end{equation}
Inserting  \eqref{step2-eq3} in \eqref{step2-eq2} and subtract  \eqref{step2-eq1} and  \eqref{step2-eq2} and taking the real part, we get 
\begin{equation}
-\Re\left(\int_0^L\tau_x\overline{y^1_x}dx\right)=\frac{o(1)}{\la^{\frac{\ell}{2}-1}}.
\end{equation}
Using the definition of $\tau$ in the above estimation, we get the desired result \eqref{u1y1}.\\
\textbf{Step 3.} In this step, we prove estimation \eqref{u3y1}. For this aim, multiplying \eqref{new-pol2} and \eqref{new-pol4}  by $\overline{y^1_{xx}}$ and $\overline{u^3_{xx}}$ respectively, integrating over $(0,L)$ and using the fact that $\frac{1}{\la}u^{3}_{xx}, \frac{1}{\la}y^1_{xx}$ are uniformly bounded in $L^2(0,L)$, we get 
\begin{equation}\label{step3-eq1}
-\la^2\int_0^Lu^3_x\overline{y^1_x}dx+\frac{E_3}{\rho_3}\int_0^Lu^3_{xx}\overline{y^3_{xx}}dx+\frac{1}{\rho_3h_3}\int_0^L\tau_x\overline{y^1_x}dx=\frac{o(1)}{\la^{\ell-2}},
\end{equation}
\begin{equation}\label{step3-eq2}
-\la^2\int_0^Ly^1_x\overline{u^3_x}dx+\frac{E_1}{\rho_1}\int_0^Ly^1_{xx}\overline{u^3_{xx}}-\frac{h_1}{2\rho_1I_1}\int_0^L\tau_x\overline{u^3_x}dx+\frac{G_1h_1}{\rho_1I_1}\int_0^Ly^1_x\overline{u^3_x}dx=\frac{o(1)}{\la^{\min{(\ell-2,\frac{\ell}{4}+\frac{1}{2})}}}. 
\end{equation}
Using the definition of $\tau$, we get 
\begin{equation}\label{step3-eq3}
\frac{1}{\rho_3h_3}\int_0^L\tau_x\overline{y^1_{x}}dx=\frac{1}{\rho_3h_3}\int_0^Lu^3_x\overline{y^1_x}dx-\frac{h_1}{2\rho_3h_3}\int_0^L\abs{y^1_x}^2dx-\frac{1}{2\rho_3}\int_0^Ly^3_x\overline{y^1_x}dx+\frac{o(1)}{\la^{\frac{\ell}{4}+\frac{1}{2}}},
\end{equation}
\begin{equation}\label{step3-eq4}
-\frac{h_1}{2\rho_1I_1}\int_0^L\tau_x\overline{u^3_x}dx=-\frac{h_1}{2\rho_1I_1}\int_0^L\abs{u^3_x}^2dx+\frac{h_1^2}{4\rho_1I_1}\int_0^Ly^1_x\overline{u^3_x}dx+\frac{h_3h_1}{4\rho_1I_1}\int_0^Ly^3_x\overline{u^3_x}dx+\frac{o(1)}{\la^{\frac{\ell}{4}+\frac{1}{2}}}.
\end{equation}
Inserting \eqref{step3-eq3} in \eqref{step3-eq1} and \eqref{step3-eq4} in \eqref{step3-eq2}, { and  taking the real part their difference}, then using the fact that $\frac{E_1}{\rho_1}=\frac{E_3}{\rho_3}$, we get \eqref{u3y1}.\\
\textbf{Step 4.} In this step, we prove estimation \eqref{y1y3}. For this aim, multiplying \eqref{new-pol4} and \eqref{new-pol5} by $\overline{y^3_{xx}}$ and $\overline{y^1_{xx}}$ respectively, integrating over $(0,L)$, and using the fact that $\frac{1}{\la}y^1_{xx},\ \frac{1}{\la}y^3_{xx}$ are uniformly bounded in $L^2(0,L)$, we get 
\begin{equation}\label{step4-eq1}
-\la^2\int_0^Ly^1_x\overline{y^3_{xx}}dx+\frac{E_1}{\rho_1}\int_0^Ly^1_{xx}\overline{y^3_{xx}}dx-\frac{h_1}{2\rho_1I_1}\int_0^L\tau_x\overline{y^3_x}dx+\frac{G_1h_1}{\rho_1I_1}\int_0^Ly^1_x\overline{y^3_x}dx=\frac{o(1)}{\la^{\min{(\ell-2,\frac{\ell}{4}+\frac{1}{2})}}},
\end{equation}
\begin{equation}\label{step4-eq2}
-\la^2\int_0^Ly^3_x\overline{y^1_x}dx+\frac{E_3}{\rho_3}\int_0^Ly^3_{xx}\overline{y^1_{xx}}dx-\frac{h_3}{2\rho_3I_3}\int_0^L\tau_x\overline{y^1_x}dx+\frac{G_3h_3}{\rho_3I_3}\int_0^Ly^3_x\overline{y^1_x}dx=\frac{o(1)}{\la^{\min{(\ell-2,\frac{\ell}{4}+\frac{1}{2})}}}.
\end{equation}
Using the definition of $\tau$ and Lemma \ref{2C2-diss}, \ref{C2-L2-est} and the fact that $y^3_x$ and $y^1_x$ are uniformly bounded in $L^2(0,L)$, we get 
\begin{equation}\label{step4-eq3}
-\frac{h_1}{2\rho_1I_1}\int_0^L\tau_x\overline{y^3_x}dx=-\frac{h_1}{2\rho_1I_1}\int_0^Lu^3_x\overline{y^3_x}dx+\frac{h_1^2}{4\rho_1I_1}\int_0^Ly^1_x\overline{y^3_x}dx+\frac{h_1h_3}{4\rho_1I_1}\int_0^L\abs{y^3_x}+\frac{o(1)}{\la^{\frac{\ell}{4}+\frac{1}{2}}}dx,
\end{equation}
\begin{equation}\label{step4-eq4}
-\frac{h_3}{2\rho_3I_3}\int_0^L\tau_x\overline{y^1_x}dx=-\frac{h_3}{2\rho_3I_3}\int_0^Lu^3_x\overline{y^1_x}dx+\frac{h_3h_1}{4\rho_3I_3}\int_0^L\abs{y^1_x}^2dx+\frac{h_3^2}{4\rho_3I_3}\int_0^Ly^3_x\overline{y^1_x}dx+\frac{o(1)}{\la^{\frac{\ell}{4}+\frac{1}{2}}}dx.
\end{equation}
Inserting \eqref{step4-eq3} in \eqref{step4-eq1} and \eqref{step4-eq4} in \eqref{step4-eq2} {and taking the real part of their difference}, then using the fact that $\frac{E_1}{\rho_1}=\frac{E_3}{\rho_3}$, we get \eqref{y1y3}. The proof has been completed.
\end{proof}
\begin{Lemma}\label{C2-L4-est}
Assume that \eqref{H3} holds. Then the solution $U\in D(\mathcal{A})$  of system \eqref{new-pol1}-\eqref{new-pol5} satisfies the following asymptotic behavior estimations 
\begin{equation}\label{C2-L4-est1}
\int_0^L\abs{u^3_x}^2dx=o(1),\ \ \int_0^L\abs{y^3_x}^2dx=o(1)\quad \text{and}\quad \int_0^L\abs{y^1_x}^2dx=o(1)
\end{equation}
and 
\begin{equation}\label{1C2-L4-est1}
\int_0^L\abs{\la u^3}^2dx=o(1),\quad \int_0^L\abs{\la y^1}^2dx=o(1)\quad \text{and}\quad \int_0^L\abs{\la y^3}^2dx=o(1).
\end{equation}
\end{Lemma}
\begin{proof}
Multiplying \eqref{C2-L3-est1} by $\frac{h_3}{2I_3}$ and summing with \eqref{C2-L3-est3}, we get 
\begin{equation}\label{C2-L4-est2}
-\frac{1}{2}\int_0^L\abs{y^3_x}^2dx-\frac{h_1}{2h_3}\Re\left(\int_0^Ly^1_x\overline{y^3_x}dx\right)+\left(\frac{1}{h_3}-\frac{G_3h_3}{I_3}\right)\Re\left(\int_0^Ly^3_x\overline{u^3_x}dx\right)=o(1).
\end{equation}
Multiplying \eqref{C2-L3-est2} by $\frac{1}{h_3}$ and summing with \eqref{C2-L4-est2}, we obtain 
\begin{equation}\label{C2-L4-est3}
\Re\left(\int_0^Lu^3_x\overline{y^3_x}dx\right)=o(1).
\end{equation}
Multiplying \eqref{u1y1} by $\frac{1}{\rho_3h_3}$ and summing with \eqref{u3y1}, then using \eqref{C2-L4-est3}, we get 
\begin{equation}\label{C2-L4-est4}
\frac{h_1}{2\rho_1I_1}\int_0^L\abs{u^3_x}^2dx-\left(\frac{h_1^2}{4\rho_1I_1}+\frac{G_1h_1}{\rho_1I_1}\right)\Re\left(\int_0^Ly^1_x\overline{u^3_x}dx\right)=o(1).
\end{equation}
Multiplying \eqref{C2-L3-est1} by $\frac{h_1}{2\rho_1I_1}$ and summing with \eqref{C2-L4-est4}, we get 
\begin{equation}\label{C2-L4-est5}
\Re\left(\int_0^Ly^1_x\overline{u^3_x}dx\right)=o(1).
\end{equation}
Using \eqref{C2-L4-est3} and \eqref{C2-L4-est5} in \eqref{C2-L3-est1}, we get 
\begin{equation}\label{C2-L4-est6}
\int_0^L\abs{u^3_x}^2dx=o(1).
\end{equation}
Multiplying \eqref{u1y1} by $\frac{h_3}{4I_3\rho_3}$ and summing with \eqref{y1y3}, then using \eqref{C2-L4-est3}, \eqref{C2-L4-est5} and \eqref{C2-L4-est6}, we get 
\begin{equation}\label{C2-L4-est7}
\frac{h_1h_3}{4I_1\rho_1}\int_0^L\abs{y^3_x}^2dx+\left(\frac{h_1^2}{4I_1\rho_1}+\frac{G_1h_1}{\rho_1I_1}-\frac{G_3h_3}{\rho_3I_3}\right)\Re\left(\int_0^Ly^1_x\overline{y^3_x}dx\right)=o(1).
\end{equation}
Multiplying \eqref{y1y3} by $-\frac{h_1}{4I_1\rho_1}$ and summing with \eqref{C2-L4-est7}, then using \eqref{C2-L4-est3} and \eqref{C2-L4-est6} we get 
\begin{equation*}
\left(\frac{G_1h_1}{\rho_1I_1}-\frac{G_3h_3}{\rho_3I_3}\right)\Re\left(\int_0^Ly^1_x\overline{y^3_x}dx\right)=o(1).
\end{equation*}
Now, using the fact that $\frac{G_1h_1}{\rho_1I_1}\neq\frac{G_3h_3}{\rho_3I_3}$ in the above estimation, we obtain 
\begin{equation}\label{C2-L4-est8}
\Re\left(\int_0^Ly^1_x\overline{y^3_x}dx\right)=o(1).
\end{equation}
Finally, inserting \eqref{C2-L4-est8}, \eqref{C2-L4-est3} and \eqref{C2-L4-est5} in \eqref{C2-L3-est2} and \eqref{u1y1}, we get the second and the third estimation in \eqref{C2-L4-est1}. In order to complete the proof of this Lemma, we need to prove \eqref{1C2-L4-est1}. For this aim, multiplying \eqref{new-pol2}, \eqref{new-pol4} and \eqref{new-pol5} respectively by $\overline{u^3}$, $\overline{y^1}$ and $\overline{y^3}$ and using Lemmas \ref{2C2-diss} and \eqref{C2-L2-est}, we get 
\begin{equation}\label{C2-L4-est9}
\left\{\begin{array}{l}
\displaystyle\int_0^L\abs{\la u^3}^2dx-\frac{E_3}{\rho_3}\int_0^L\abs{u^3_x}^2dx=o(1),\\[0.1in]
\displaystyle
\int_0^L\abs{\la y^1}^2dx-\frac{E_1}{\rho_1}\int_0^L\abs{y^1_x}^2dx=o(1),\\[0.1in]
\displaystyle
\int_0^L\abs{\la y^3}^2dx-\frac{E_3}{\rho_3}\int_0^L\abs{y^3_x}^2dx=o(1).
\end{array}
\right.
\end{equation}
Using \eqref{C2-L4-est1} in \eqref{C2-L4-est9}, we get \eqref{1C2-L4-est1}. The proof has been completed. 
\end{proof}
\begin{Lemma}\label{C2-L5-est}
Assume that $G_1\neq G_3$ and (\eqref{H4} or \eqref{H5} holds). Then the solution $U\in D(\mathcal{A})$  of system \eqref{new-pol1}-\eqref{new-pol5} satisfies the following asymptotic behavior estimations 
\begin{equation}\label{C2-L5-est1}
h_2\int_0^L\abs{u_x^3}^2dx+h_1\left(G_1-\frac{h_2}{2}\right)\Re\left(\int_0^Ly^1_x\overline{u^3_x}dx\right)+h_3\left(G_3-\frac{h_2}{2}\right)\Re\left(\int_0^Ly^3_x\overline{u_x^3}dx\right)=o(1).
\end{equation}
\begin{equation}\label{C2-L5-est2}
\left(G_1h_1-\frac{h_1h_2}{2}\right)\int_0^L\abs{y^1_x}^2dx+\left(G_3h_3-\frac{h_3h_2}{2}\right)\Re\left(\int_0^Ly^3_x\overline{y^1_x}dx\right)+h_2\Re\left(\int_0^Lu^3_x\overline{y^1_x}dx\right)=o(1).
\end{equation}
\begin{equation}\label{C2-L5-est3}
\left(G_3h_3-\frac{h_2h_3}{2}\right)\int_0^L\abs{y^3_x}^2dx+\left(G_1h_1-\frac{h_1h_2}{2}\right)\Re\left(\int_0^Ly^1_x\overline{y^3_x}dx\right)+h_2\Re\left(\int_0^Lu^3_x\overline{y^3_x}dx\right)=o(1).
\end{equation}
\end{Lemma}
\begin{proof}
The proof of this Lemma is divided into two steps.\\
\textbf{Step 1.} In order to prove \eqref{C2-L5-est1}-\eqref{C2-L5-est3}, we need to prove the following estimations 
\begin{equation}\label{C2-L5-est4}
\|\omega_{xxx}\|\leq \frac{o(1)}{\la},\ \ \abs{u^3_x(\xi)},\abs{y^i_{x}(\xi)}\leq O(\la^{\frac{1}{2}}),\ \ \abs{\omega_{xxx}(\xi)}\leq \frac{o(1)}{\la^{\frac{1}{2}}}\ \ \text{for}\ i\in \{1,3\}\ \ \text{and}\ \ \xi\in \{0,L\}.
\end{equation}
Using Galgliardo-Nirenberg interpolation inequality, Lemma \ref{C2-L2-est} and  the fact that $\omega_{xxxx}$ is uniformly bounded in $L^2(0,L)$ and   taking $\ell=6$, we get 
\begin{equation*}
\|\omega_{xxx}\|\leq \|\omega_{xxxx}\|^{\frac{1}{2}}\|\omega_{xx}\|^{\frac{1}{2}}+\|\omega_{xx}\|\leq \frac{o(1)}{\la}.
\end{equation*}
Since $\omega\in H^4(0,L)$, then $\omega_{xxx}\in H^1(0,L)$. Using Gagliardo inequality, we get 
\begin{equation*}
|\omega_{xxx}(\xi)|\leq \|\omega\|_{H^4(0,L)}^{\frac{3}{4}}\|\omega\|^{\frac{1}{4}}_{H^2(0,L)}\leq \frac{o(1)}{\la^{\frac{1}{2}}}.
\end{equation*}
Since $u^3\in H^2(0,L)$, then 
\begin{equation*}
|u^3_x(\xi)|\leq \|u^3\|^{\frac{1}{2}}_{H^1(0,L)}\|u^3\|^{\frac{1}{2}}_{H^2(0,L)}\leq O(\la^{\frac{1}{2}}).
\end{equation*}
Using the same argument, we obtain that $\abs{y^i(x)(\xi)}\leq O(\la^{\frac{1}{2}})$ for $i\in {1,3}$ and $\xi\in \{0,L\}$.\\
\textbf{Step 2.} In this step, we need to prove estimations \eqref{C2-L5-est1}-\eqref{C2-L5-est3}. For this aim, multiplying \eqref{new-pol3} by $\overline{u^3_{x}}$, integrating over $(0,L)$, taking the real part and using the fact that $u^3_x$ is uniformly bounded in $L^2(0,L)$, we get 
\begin{equation}\label{C2-L5-est5}
\left\{\begin{array}{l}
\displaystyle
\frac{EI}{\rho h}\left(\Re\left(\int_0^L\omega_{xxx}\overline{u^3_{xx}}dx\right)-\Re\left(\left[\omega_{xxx}(\xi)\overline{u^3_x}(\xi)\right]_{\xi=0}^{\xi=L}\right)\right)\\
\displaystyle
+\frac{G_1h_1}{\rho h}\Re\left(\int_0^Ly^1_x\overline{u^3_x}dx\right)+\frac{G_3h_3}{\rho h}\Re\left(\int_0^Ly^3_x\overline{u^3_x}dx\right)+\frac{h_2}{\rho h}\Re\left(\int_0^L\tau_x\overline{u^3_x}dx\right)=\frac{o(1)}{\la^2}.
\end{array}
\right.
\end{equation}
Using Step 1 and the fact that $\frac{1}{\la}\overline{u^3_{xx}}$ is uniformly bounded in $L^2(0,L)$,  we obtain 
\begin{equation*}
\left|\frac{EI}{\rho h}\left(\Re\left(\int_0^L\omega_{xxx}\overline{u^3_{xx}}dx\right)-\Re\left(\left[\omega_{xxx}(\xi)\overline{u^3_x}(\xi)\right]_{\xi=0}^{\xi=L}\right)\right)\right|=o(1).
\end{equation*}
Inserting the above estimation in \eqref{C2-L5-est5}, we get 
\begin{equation}\label{C2-L5-est6}
G_1h_1\Re\left(\int_0^Ly^1_x\overline{u^3_x}dx\right)+G_3h_3\Re\left(\int_0^Ly^3_x\overline{u^3_x}dx\right)+h_2\Re\left(\int_0^L\tau_x\overline{u^3_x}dx\right)=o(1).
\end{equation}
Now, using the definition of the function $\tau$, we get 
\begin{equation*}
h_2\Re\left(\int_0^L\tau_x\overline{u^3_x}dx\right)=h_2\int_0^L\abs{u^3_x}^2dx-\frac{h_1h_2}{2}\Re\left(\int_0^Ly^1_x\overline{u^3_x}dx\right)-\frac{h_2h_3}{2}\Re\left(\int_0^Ly^3_x\overline{u^3_x}dx\right)+\frac{o(1)}{\la^2}.
\end{equation*}
Inserting the above equation in \eqref{C2-L5-est6}, we get \eqref{C2-L5-est1}. Now, multiplying \eqref{new-pol3} by $y^1_x$ and $y^3_x$ respectively and using the same argument in Step 1, we get \eqref{C2-L5-est2} and \eqref{C2-L5-est3}.
\end{proof}
\begin{Lemma}\label{C2-L6-est}
Assume that $G_1\neq G_3$ and (\eqref{H4} or \eqref{H5} holds). Then the solution $U\in D(\mathcal{A})$  of system \eqref{new-pol1}-\eqref{new-pol5} satisfies the following asymptotic behavior estimations 
\begin{equation}\label{C2-L6-est1}
\int_0^L\abs{u^3_x}^2dx=o(1),\ \int_0^L\abs{y^1_x}^2dx=o(1), \quad \int_0^L\abs{y^3_x}^2dx=o(1), 
\end{equation}
and 
\begin{equation}\label{1C2-L6-est1}
\int_0^L\abs{\la u^3}^2dx=o(1),\quad \int_0^L\abs{\la y^1}^2dx=o(1), \quad \int_0^L\abs{\la y^3}^2dx=o(1).
\end{equation}
\end{Lemma}
\begin{proof}
The proof of this Lemma is divided into several steps.\\
\textbf{Step 1.} The aim of this step is to prove that 
\begin{equation}\label{C2-L6-est2}
\Re\left(\int_0^Lu^3_x\overline{y^3_x}dx\right)=\left\{\begin{array}{lll}\displaystyle\frac{o(1)}{\la^2}&\text{if}&\eqref{H4}\ \text{holds},\\ o(1)&\text{if}&\eqref{H5}\ \text{holds}.\end{array}\right.
\end{equation}
Using the same technique in Lemma \ref{C2-L4-est}, Lemma \ref{C2-L2-est} and using the fact that $u^3_x$ is uniformly bounded, we obtain the aim of this step.\\ 
\textbf{Step 2.} The aim of this step is to prove that 
\begin{equation}\label{C2-L6-est3}
\Re\left(\int_0^Lu^3_x\overline{y^1_x}dx\right)=o(1)\quad \text{and}\quad \int_0^L\abs{u^3_x}^2dx=o(1).
\end{equation}
Multiplying \eqref{C2-L3-est1} by $h_2$ and summing with \eqref{C2-L5-est1}, then using step 1, we get the the first estimation in \eqref{C2-L6-est3}. Now, Inserting  the first estimation in \eqref{C2-L6-est3} and \eqref{C2-L6-est2} in \eqref{C2-L3-est1}, we get the second estimation in \eqref{C2-L6-est3}.\\
\textbf{Step 3.} In this step we prove that 
\begin{equation}\label{C2-L6-est4}
\int_0^L\abs{y^1_x}^2dx=o(1)\quad \text{and}\quad \int_0^L\abs{y^3_x}^2dx=o(1).
\end{equation}
For this aim, multiplying equation \eqref{C2-L3-est2} by $h_2$ and summing with \eqref{C2-L5-est3} then using Step 1 and Step 2, we get 
\begin{equation}\label{C2-L6-est5}
G_3h_3\int_0^L\abs{y^3_x}^2dx+G_1h_1\Re\left(\int_0^Ly^1_x\overline{y^3_x}dx\right)=o(1).
\end{equation}
Now, Multiplying \eqref{u1y1} by $h_1$ and \eqref{C2-L3-est2} by $-h_3$, summing the  result, we get 
\begin{equation}\label{C2-L6-est6}
h_1^2\int_0^L\abs{y^1_x}^2dx-h_3^2\int_0^L\abs{y^3_x}^2dx=\frac{o(1)}{\la}.
\end{equation}
Now, multiplying \eqref{u1y1} by $h_2$ and summing with \eqref{C2-L5-est2}, we get 
\begin{equation}\label{C2-L6-est7}
G_1h_1\int_0^L\abs{y^1_x}^2dx+G_3h_3\Re\left(\int_0^Ly^3_x\overline{y^1_x}dx\right)=o(1).
\end{equation}
{Similarly, multiplying} \eqref{C2-L6-est5} by $G_3h_3$ and \eqref{C2-L6-est7} by $-G_1h_1$, then summing the result we get
\begin{equation}\label{C2-L6-est8}
G_3^2h_3^2\int_0^L\abs{y^3_x}^2dx-G_1^2h_1^2\int_0^L\abs{y^1_x}^2dx=o(1).
\end{equation}
{Furthermore, multiplying} \eqref{C2-L6-est6} by $G_3^2$ and summing with \eqref{C2-L6-est8}, we get 
\begin{equation}\label{C2-L6-est9}
\left(G_3^2-G_1^2\right)h_1^2\int_0^L\abs{y^1_x}^2dx=o(1).
\end{equation}
Using the fact that $G_3\neq G_1$, then we obtain 
\begin{equation}\label{C2-L6-est10}
h_1^2\int_0^L\abs{y^1_x}^2dx=o(1).
\end{equation}
Inserting \eqref{C2-L6-est10} in \eqref{C2-L6-est6}, we get the second estimation in \eqref{C2-L6-est3}. Next, in order to prove \eqref{1C2-L6-est1}, multiplying \eqref{new-pol2}, \eqref{new-pol4} and \eqref{new-pol5} respectively by $\overline{u^3}$, $\overline{y^1}$ and $\overline{y^3}$, we get 
\begin{equation}\label{C2-L6-est11}
\left\{\begin{array}{l}
\displaystyle\int_0^L\abs{\la u^3}^2dx-\frac{E_3}{\rho_3}\int_0^L\abs{u^3_x}^2dx=o(1),\\[0.1in]
\displaystyle
\int_0^L\abs{\la y^1}^2dx-\frac{E_1}{\rho_1}\int_0^L\abs{y^1_x}^2dx=o(1),\\[0.1in]
\displaystyle
\int_0^L\abs{\la y^3}^2dx-\frac{E_3}{\rho_3}\int_0^L\abs{y^3_x}^2dx=o(1).
\end{array}
\right.
\end{equation}
Using \eqref{C2-L6-est1} in \eqref{C2-L6-est11}, we get \eqref{1C2-L6-est1}. The proof has been completed. 
\end{proof}
\\
{We are now ready to finish the proof of Theorem \ref{POL-ac}}. We distinguish two cases.\\
\textbf{Case 1}. If $G_1\neq G_3$ and \eqref{H3} holds, then using  Lemmas \ref{2C2-diss}, \ref{C2-L2-est}, \ref{C2-L4-est}, we get 
$$
\begin{array}{lllll}
\displaystyle
\int_0^L\abs{v^1}^2dx=\frac{o(1)}{\la^2},&\displaystyle 
\int_0^L\abs{z^1}^2dx=o(1),&\displaystyle
\int_0^L\abs{u_x^1}^2dx=\frac{o(1)}{\la^2},&\displaystyle \int_0^L\abs{u_x^3}^2dx=o(1),&\displaystyle \int_0^L\abs{\psi}^2dx=\frac{o(1)}{\la^2},\\[0.1in]
\displaystyle\int_0^L\abs{\omega_{xx}}^2dx=\frac{o(1)}{\la^2},&\displaystyle 
\int_0^L\abs{y^1_x}^2dx=o(1),&\displaystyle  \int_0^L\abs{y^3_x}^2dx=o(1),&\displaystyle \int_0^L\abs{v^3}^2dx=o(1),&\displaystyle\int_0^L\abs{z^3}^2dx=o(1).
\end{array}
$$
It follows that {$\|U\|_{\mathcal{H}}=o(1)$}, which contradicts ${\rm (H_2)}$. This implies that 
$$
\sup_{\la\in \mathbb{R}}\|(i\la I-\mathcal{A})^{-1}\|_{\mathcal{H}}{\le}O(\la^{2}).
$$
\textbf{Case 2}. If $G_1\neq G_3$ and (\eqref{H4} or \eqref{H5} holds), then  using  Lemmas \ref{2C2-diss}, \ref{C2-L2-est}, \ref{C2-L6-est}, we get 
$$
\begin{array}{lllll}
\displaystyle
\int_0^L\abs{v^1}^2dx=\frac{o(1)}{\la^6},&\displaystyle 
\int_0^L\abs{z^1}^2dx=o(1),&\displaystyle
\int_0^L\abs{u_x^1}^2dx=\frac{o(1)}{\la^6},&\displaystyle \int_0^L\abs{u_x^3}^2dx=o(1),&\displaystyle \int_0^L\abs{\psi}^2dx=\frac{o(1)}{\la^6},\\[0.1in]
\displaystyle\int_0^L\abs{\omega_{xx}}^2dx=\frac{o(1)}{\la^4},&\displaystyle 
\int_0^L\abs{y^1_x}^2dx=o(1),&\displaystyle  \int_0^L\abs{y^3_x}^2dx=o(1),&\displaystyle \int_0^L\abs{v^3}^2dx=o(1),&\displaystyle\int_0^L\abs{z^3}^2dx=o(1).
\end{array}
$$
It follows that {$\|U\|_{\mathcal{H}}=o(1)$}, which contradicts ${\rm (H_2)}$. This implies that 
$$
\sup_{\la\in \mathbb{R}}\|(i\la I-\mathcal{A})^{-1}\|_{\mathcal{H}}{\le}O(\la^{6}).
$$
The proof has been completed. 
\subsection{Proof of Theorem \ref{POL-bc}}. In this subsection, we assume that $(b,c>0\ \text{and}\ a=0\ \text{and}\ \ \frac{E_1}{\rho_1}\neq \frac{E_3}{\rho_3})$. We will check the condition ${\rm (H_2)}$ by finding a contradiction with \eqref{1pol1}. From \eqref{pol1}-\eqref{pol10}, we obtain the following system 
\begin{eqnarray}
\la^2u^1+(\rho_1h_1)^{-1}\left[E_1h_1u^1_{xx}+\tau\right]&=&-\la^{-6}f_2-i\la^{-5}f_1,\label{3new-pol1}\\
\la^2y^1+(\rho_1I_1)^{-1}\left[E_1I_1y^1_{xx}+\frac{h_1}{2}\tau-G_1h_1(\omega_x+y^1)-i\la by^1\right]&=&-\la^{-6}f_4-\la^{-6}(\frac{b}{\rho_1I_1}+i\la)f_3,\label{3new-pol4}\\
\hspace{0.75cm}\la^2\omega+(\rho h)^{-1}\left[-EI\omega_{xxxx}+G_h\omega_{xx}+G_1h_1y^1_x+G_3h_3y^3_x+h_2\tau_x-i\la c\omega\right]&=&-\la^{-6}f_6-\la^{-6}(c+i\la)f_5,\label{3new-pol3}\\
\la^2u^3+(\rho_3 h_3)^{-1}\left[E_3h_3u^3_{xx}-\tau\right]&=&-\la^{-6}f_8-i\la^{-5}f_7,\label{3new-pol2}\\
\la^2y^3+(\rho_3I_3)^{-1}\left[E_3I_3y_{xx}^3+\frac{h_3}{2}\tau-G_3h_3(\omega_x+y^3)\right]&=&-\la^{-6}f_{10}-i\la^{-5}f_9,\label{3new-pol5}
\end{eqnarray}
where $G_h=G_1h_1+G_3h_3$. For clarity, we divide the proof into several Lemmas.
\begin{Lemma}\label{3C3-diss}
The solution $U\in D(\mathcal{A})$ of system \eqref{3new-pol1}-\eqref{3new-pol5} satisfies the following asymptotic behavior estimations 
\begin{equation}\label{Case2-diss}
\int_0^L\abs{z^1}^2dx=\frac{o(1)}{\la^{6}},\ \ \int_0^L\abs{\psi}^2dx=\frac{o(1)}{\la^{6}},\ \ \int_0^L\abs{y^1}^2dx=\frac{o(1)}{\la^{8}}\ \text{and}\ \int_0^L\abs{\omega}^2dx=\frac{o(1)}{\la^{8}}.
\end{equation}\label{C3-diss}
\end{Lemma}
\begin{proof}
Same arguments of Lemma \ref{2C2-diss}.
\end{proof}
\begin{Lemma}\label{C3-L2-est}
The solution $U\in D(\mathcal{A})$  of system \eqref{3new-pol1}-\eqref{3new-pol5} satisfies the following asymptotic behavior estimations 
\begin{equation}\label{C2-L2-est1}
\int_0^L\abs{\omega_{xx}}^2=\frac{o(1)}{\la^4},\ \int_0^L\abs{\omega_{x}}^2=\frac{o(1)}{\la^4}\quad \text{and}\quad \int_0^L\abs{y^1_x}^2dx=\frac{o(1)}{\la^{5}}.
\end{equation}
\end{Lemma}
\begin{proof}
For the first two estimations, using the same arguments of Lemma \ref{C2-L2-est}. Now, we will prove the third estimation. For this aim, multiplying \eqref{3new-pol4} by $\overline{y^1}$, integrating over $(0,L)$  using Lemma \ref{3C3-diss} and the facts that $\la \tau$ is uniformly bounded in $L^2(0,L)$, $f_3\to 0$ in $H_0^1(0,L)$ and $f_4\to 0$ in $L^2(0,L)$, we get the third estimation in \eqref{C2-L2-est1}.
\end{proof}
\begin{Lemma}\label{C3-Lemma3}
Assume that $\frac{E_1}{\rho_1}\neq \frac{E_3}{\rho_3}$. Then, the solution $U\in D(\mathcal{A})$ of system \eqref{3new-pol1}-\eqref{3new-pol5} satisfies the following asymptotic behavior estimations 
\begin{equation}\label{1C3-Lemma3}
-\int_0^L\abs{u^1_x}^2dx+\Re\left(\int_0^Lu^3_x\overline{u^1_x}dx\right)-\frac{h_3}{2}\Re\left(\int_0^Ly^3_x\overline{u^1_x}dx\right)=\frac{o(1)}{\la^2},
\end{equation}
\begin{equation}\label{2C3-Lemma3}
\frac{h_3}{2}\int_0^L\abs{y^3_x}^2dx-\Re\left(\int_0^Lu^3_x\overline{y^3_x}dx\right)+\Re\left(\int_0^Lu^1_x\overline{y^3_x}dx\right)=\frac{o(1)}{\la^{\frac{1}{2}}},
\end{equation}
\begin{equation}\label{3C3-Lemma3}
\int_0^L\abs{u^3_x}^2dx-\Re\left(\int_0^Lu^1_x\overline{u^3_x}dx\right)-\frac{h_3}{2}\Re\left(\int_0^Ly^3_x\overline{u^3_x}dx\right)=\frac{o(1)}{\la^{\frac{1}{2}}},
\end{equation}
\begin{equation}\label{4C3-Lemma3}
\left\{\begin{array}{l}
\displaystyle
\frac{h_3}{2I_3}\int_0^L\abs{u^3_x}^2dx-\frac{1}{2}\int_0^L\abs{y^3_x}^2dx-\frac{1}{h_3}\Re\left(\int_0^Lu^1_x\overline{y^3_x}dx\right)-\frac{h_3}{2I_3}\Re\left(\int_0^Lu^1_x\overline{u^3_x}dx\right)\\[0.1in]
\displaystyle
+\left(\frac{1}{h_3}-\frac{h_3^2}{4I_3}-\frac{G_3h_3}{I_3}\right)\Re\left(\int_0^Lu^3_x\overline{y^3_x}dx\right)=\frac{o(1)}{\la^2}.
\end{array}
\right.
\end{equation}
\end{Lemma}
\begin{proof}
The proof of this Lemma is divided into several steps.\\
\textbf{Step 1.} In this step, we prove estimation \eqref{1C3-Lemma3}. For this aim, multiply \eqref{3new-pol1} and \eqref{3new-pol4}  by $\overline{y^1_{xx}}$ and $\overline{u^1_{xx}}$ respectively,  integrating by parts over $(0,L)$,  subtracting the two equations, taking the real part, and using the facts that $\tau_x$ is uniformly bounded in $L^2(0,L)$ and using Lemmas \ref{3C3-diss} , \ref{C3-L2-est}, we get 
\begin{equation}\label{5C3-Lemma3}
\Re\left(\int_0^L\tau_x\overline{u^1_x}dx\right)=\frac{o(1)}{\la^{2}}.
\end{equation}
Using the definition of the function $\tau$ and Lemma \ref{C3-L2-est} in \eqref{5C3-Lemma3}, we get \eqref{1C3-Lemma3}.\\
\textbf{Step 2.} In this step, we prove estimation \eqref{2C3-Lemma3}. For this aim, multiplying \eqref{3new-pol4} and \eqref{3new-pol5}  by $\frac{E_3}{\rho_3}\overline{y^3_{xx}}$ and $\frac{E_1}{\rho_1}\overline{y^1_{xx}}$ respectively, integrating by parts over $(0,L)$,  subtracting their results, taking the real part, and using the facts that $\frac{E_1}{\rho_1}\neq \frac{E_3}{\rho_3}$ and using Lemmas \ref{3C3-diss} and \ref{C3-L2-est}, we get 
\begin{equation}\label{6C3-Lemma3}
-\Re\left(\int_0^L\tau_x\overline{y^3_x}dx\right)=\frac{o(1)}{\la^{\frac{1}{2}}}.
\end{equation}
Using the definition of the function $\tau$ in \eqref{6C3-Lemma3}, we get the \eqref{2C3-Lemma3}.\\
\textbf{Step 3.} In this step, we prove estimation \eqref{3C3-Lemma3}. For this aim, multiplying \eqref{3new-pol2} and \eqref{3new-pol4}  by $\frac{E_1}{\rho_1}\overline{y^1_{xx}}$ and $\frac{E_3}{\rho_3}\overline{u^3_{xx}}$ respectively,  integrating by parts on $(0,L)$,  subtracting their results, taking the real part, and using the fact that $\frac{E_1}{\rho_1}\neq \frac{E_3}{\rho_3}$ and using Lemmas \ref{3C3-diss}, \ref{C3-L2-est}, we get 
\begin{equation}\label{7C3-Lemma3}
\Re\left(\int_0^L\tau_x\overline{u^3_x}dx\right)=\frac{o(1)}{\la^{\frac{1}{2}}}.
\end{equation}
Using the definition of the function $\tau$ in \eqref{7C3-Lemma3}, we get estimation \eqref{3C3-Lemma3}.\\
\textbf{Step 4.} In order to complete the proof of this Lemma, we prove estimation \eqref{4C3-Lemma3}. For this aim, multiply \eqref{3new-pol2} and \eqref{3new-pol5}  by $\overline{y^3_{xx}}$ and $\overline{u^3_{xx}}$ respectively,  integrating by parts on $(0,L)$,  subtracting their results, taking the real part, and using Lemmas \ref{3C3-diss} and \ref{C3-L2-est}, we get 
\begin{equation}\label{8C3-Lemma3}
\frac{1}{h_3}\Re\left(\int_0^L\tau_x\overline{y^3_x}dx\right)+\frac{h_3}{2I_3}\Re\left(\int_0^L\tau_x\overline{u^3_x}dx\right)-\frac{G_3h_3}{I_3}\Re\left(\int_0^Ly^3_x\overline{u^3_x}dx\right)=\frac{o(1)}{\la^2}.
\end{equation}
Using the definition of $\tau$ in \eqref{8C3-Lemma3}, we get estimation \eqref{4C3-Lemma3}. The proof has been completed.
\end{proof}
\begin{Corollary}
Assume that $\frac{E_1}{\rho_1}\neq \frac{E_3}{\rho_3}$. Then, the solution $U\in D(\mathcal{A})$ of system \eqref{3new-pol1}-\eqref{3new-pol5} satisfies the following asymptotic behavior estimations 
\begin{equation}\label{9C3-Lemma3}
\Re\left(\int_0^Lu^3_x\overline{y^3_x}dx\right)=\frac{o(1)}{\la^{\frac{1}{2}}}.
\end{equation}
\end{Corollary}
\begin{proof}
Multiplying \eqref{2C3-Lemma3} by $h_3^{-1}$ and summing with \eqref{4C3-Lemma3}, we get 
\begin{equation}\label{10C3-Lemma3}
\frac{h_3}{2I_3}\int_0^L\abs{u^3_x}^2dx-\frac{h_3}{2I_3}\Re\left(\int_0^Lu^1_x\overline{u^3_x}dx\right)-\left(\frac{h_3^2}{4I_3}+\frac{G_3h_3}{I_3}\right)\Re\left(\int_0^Lu^3_x\overline{y^3_x}dx\right)=\frac{o(1)}{\la^{\frac{1}{2}}}.
\end{equation}
Now, multiplying \eqref{3C3-Lemma3} by $-\frac{h_3}{2I_3}$ and summing with \eqref{10C3-Lemma3}, we get \eqref{9C3-Lemma3}. The proof has been completed. 
\end{proof}
\begin{Lemma}\label{C3-Lemma4}
Assume that $\frac{E_1}{\rho_1}\neq \frac{E_3}{\rho_3}$. Then, the solution $U\in D(\mathcal{A})$ of system \eqref{3new-pol1}-\eqref{3new-pol5} satisfies the following asymptotic behavior estimations 
\begin{equation}\label{1C3-Lemma4}
-h_2\int_0^L\abs{u^1_x}^2dx+h_2\Re\left(\int_0^Lu^3_x\overline{u^1_x}dx\right)+h_3\left(G_3-\frac{h_2}{2}\right)\Re\left(\int_0^Ly^3_x\overline{u^1_x}dx\right)=o(1).
\end{equation}
\end{Lemma}
\begin{proof}
First, using the same arguments in Step 1 of  Lemma \ref{C2-L5-est1}, we obtain 
\begin{equation}\label{2C3-Lemma4}
\|\omega_{xxx}\|\leq \frac{o(1)}{\la},\quad \abs{u^1_x(0)},\abs{u^1_x(L)}\leq O\left(\la^{\frac{1}{2}}\right)\quad \text{and}\quad \abs{\omega_{xxx}(0)},\abs{\omega_{xxx}(L)}\leq \frac{o(1)}{\la^{\frac{1}{2}}}. 
\end{equation}
Next, multiplying \eqref{3new-pol3} by $\overline{u^1_x}$, integrating over $(0,L)$ and using the fact that $u^1_x$ is uniformly bounded in $L^2(0,L)$, we get 
\begin{equation}\label{3C3-Lemma4}
\left\{\begin{array}{l}
\displaystyle
\frac{EI}{\rho h}\left(\Re\left(\int_0^L\omega_{xxx}\overline{u^1_{xx}}dx\right)-\Re\left(\left[\omega_{xxx}(\xi)\overline{u^1_x}(\xi)\right]_{\xi=0}^{\xi=L}\right)\right)\\
\displaystyle
+\frac{G_3h_3}{\rho h}\Re\left(\int_0^Ly^3_x\overline{u^3_x}dx\right)+\frac{h_2}{\rho h}\Re\left(\int_0^L\tau_x\overline{u^3_x}dx\right)=\frac{o(1)}{\la^2}.
\end{array}
\right.
\end{equation}
Using \eqref{2C3-Lemma4} and the fact that $\frac{1}{\la}\overline{u^3_{xx}}$ is uniformly bounded in $L^2(0,L)$,  we obtain 
\begin{equation*}
\left|\frac{EI}{\rho h}\left(\Re\left(\int_0^L\omega_{xxx}\overline{u^1_{xx}}dx\right)-\Re\left(\left[\omega_{xxx}(\xi)\overline{u^1_x}(\xi)\right]_{\xi=0}^{\xi=L}\right)\right)\right|=o(1).
\end{equation*}
Inserting the above estimation in \eqref{3C3-Lemma4}, we get 
\begin{equation}\label{4C3-Lemma4}
G_3h_3\Re\left(\int_0^Ly^3_x\overline{u^1_x}dx\right)+h_2\Re\left(\int_0^L\tau_x\overline{u^1_x}dx\right)=o(1).
\end{equation}
Using the definition of $\tau$ in \eqref{4C3-Lemma4}, we get estimation \eqref{1C3-Lemma4}. The proof has been completed.
\end{proof}
\begin{Corollary}
Assume that $\frac{E_1}{\rho_1}\neq \frac{E_3}{\rho_3}$. Then, the solution $U\in D(\mathcal{A})$ of system \eqref{3new-pol1}-\eqref{3new-pol5} satisfies the following asymptotic behavior estimations 
\begin{equation}\label{COR1}
\Re\left(\int_0^Ly^3_x\overline{u^1_x}dx\right)=o(1)\quad \text{and}\quad \int_0^L\abs{y^3_x}^2dx=o(1).
\end{equation}
\end{Corollary}
\begin{proof}
Multiplying \eqref{1C3-Lemma3} by $-h_2$ and summing with \eqref{1C3-Lemma4}, we get the first estimation in \eqref{COR1}. Now, using the first estimation in \eqref{COR1} and \eqref{9C3-Lemma3} in \eqref{2C3-Lemma3}, we get the second estimation in \eqref{COR1}. The proof has been completed. 
\end{proof}
\begin{Lemma}\label{C3-Lemma5}
Assume that $\frac{E_1}{\rho_1}\neq \frac{E_3}{\rho_3}$. Then, the solution $U\in D(\mathcal{A})$ of system \eqref{3new-pol1}-\eqref{3new-pol5} satisfies the following asymptotic behavior estimation
\begin{equation}\label{C3-Lemma5-1}
\Re\left(\int_0^Lu^1_x\overline{u^3_x}dx\right)=\frac{o(1)}{\la^2}.
\end{equation}
\end{Lemma}
\begin{proof}
The proof of this Lemma is divided into two steps.\\
\textbf{Step 1.} In this step, we prove the following estimation. 
\begin{equation}\label{1C3-Lemma5-1}
\left\{\begin{array}{l}
\displaystyle
\rho_1\rho_3\la^2\left(\frac{E_1}{\rho_1}-\frac{E_3}{\rho_3}\right)\Re\left(\int_0^Lu^1_x\overline{u^3_x}dx\right)-\frac{E_3}{h_1}\int_0^L\abs{u^3_x}^2dx+\frac{E_1}{h_3}\int_0^L\abs{u^1_x}^2dx\\
\displaystyle
\left(\frac{E_3}{h_1}-\frac{E_1}{h_3}\right)\Re\left(\int_0^Lu^1_x\overline{u^3_x}dx\right)=o(1).
\end{array}
\right.
\end{equation}
For this aim, multiplying equations \eqref{3new-pol1} and \eqref{3new-pol2}   by $\frac{E_3}{\rho_3}\overline{u^3_{xx}}$ and $\frac{E_1}{\rho_1}\overline{u^1_{xx}}$ respectively,  integrating by parts over $(0,L)$ and using the facts that $\la^{-1}u^{3}_{xx}$ and $\la^{-1}u^1_{xx}$ are uniformly bounded in $L^2(0,L)$, we get 
\begin{equation}\label{C3-Lemma5-2}
-\frac{E_3}{\rho_3}\la^2\int_0^Lu^1_x\overline{u^3_x}dx+\frac{E_1E_3}{\rho_1\rho_3}\int_0^Lu^1_{xx}\overline{u^3_{xx}}dx-\frac{E_3}{\rho_1\rho_3h_1}\int_0^L\tau_x\overline{u^3_x}dx=\frac{o(1)}{\la^4}
\end{equation}
and
\begin{equation}\label{C3-Lemma5-3}
-\frac{E_1}{\rho_1}\la^2\int_0^Lu^3_x\overline{u^1_{x}}dx+\frac{E_3E_1}{\rho_3\rho_1}\int_0^Lu^3_{xx}\overline{u^1_{xx}}dx+\frac{E_1}{\rho_1\rho_3h_3}\int_0^L\tau_x\overline{u^1_x}dx=\frac{o(1)}{\la^4}.
\end{equation}
Subtracting \eqref{C3-Lemma5-2} and \eqref{C3-Lemma5-3} and taking the real part, we get 
\begin{equation}\label{C3-Lemma5-4}
\rho_1\rho_3\la^2\left(\frac{E_1}{\rho_1}-\frac{E_3}{\rho_3}\right)\Re\left(\int_0^Lu^1_x\overline{u^3_x}dx\right)-\frac{E_3}{h_1}\Re\left(\int_0^L\tau_x\overline{u^3_x}dx\right)-\frac{E_1}{h_3}\Re\left(\int_0^L\tau_x\overline{u^1_x}dx\right)=\frac{o(1)}{\la^{4}}.
\end{equation}
Using the definition of $\tau$, Lemma \ref{C3-L2-est} , \eqref{9C3-Lemma3}, \eqref{COR1} and the fact that $u^1_{x}$ and $u^3_x$ are uniformly bounded in $L^2(0,L)$, we get
\begin{equation}\label{C3-Lemma5-5}
-\frac{E_3}{h_1}\Re\left(\int_0^L\tau_x\overline{u^3_x}dx\right)=\frac{E_3}{h_1}\Re\left(\int_0^Lu^1_x\overline{u^3_x}dx\right)-\frac{E_3}{h_1}\int_0^L\abs{u^3_x}^2dx+o(1)
\end{equation}
and 
\begin{equation}\label{C3-Lemma5-6}
-\frac{E_1}{h_3}\Re\left(\int_0^L\tau_x\overline{u^1_x}dx\right)=\frac{E_1}{h_3}\int_0^L\abs{u^1_x}^2dx-\frac{E_1}{h_3}\Re\left(\int_0^Lu^3_x\overline{u^1_x}dx\right)+o(1).
\end{equation}
Inserting \eqref{C3-Lemma5-5} and \eqref{C3-Lemma5-6} in \eqref{C3-Lemma5-4}, we get the desired result \eqref{1C3-Lemma5-1}.\\
\textbf{Step 2.} In this part we prove \eqref{C3-Lemma5-1}. For this aim, multiplying \eqref{1C3-Lemma3} by $\frac{E_1}{h_3}$, summing with \eqref{1C3-Lemma5-1}  and using \eqref{COR1}, we get 
\begin{equation}\label{C3-Lemma5-7}
\rho_1\rho_3\la^2\left(\frac{E_1}{\rho_1}-\frac{E_3}{\rho_3}\right)\Re\left(\int_0^Lu^1_x\overline{u^3_x}dx\right)-\frac{E_3}{h_1}\int_0^L\abs{u^3_x}^2dx+\frac{E_3}{h_1}\Re\left(\int_0^Lu^1_x\overline{u^3_x}dx\right)=o(1).
\end{equation}
Now, multiplying \eqref{3C3-Lemma3} by $\frac{E_3}{h_1}$ summing with \eqref{C3-Lemma5-7} and using \eqref{9C3-Lemma3}, we get 
\begin{equation*}
\rho_1\rho_3\la^2\left(\frac{E_1}{\rho_1}-\frac{E_3}{\rho_3}\right)\Re\left(\int_0^Lu^1_x\overline{u^3_x}dx\right)=o(1).
\end{equation*}
Using the fact that $\frac{E_1}{\rho_1}\neq \frac{E_3}{\rho_3}$, we get the desired result \eqref{C3-Lemma5-1}. The proof has been completed. 
\end{proof}
\begin{Corollary}
Assume that $\frac{E_1}{\rho_1}\neq \frac{E_3}{\rho_3}$. Then, the solution $U\in D(\mathcal{A})$ of system \eqref{3new-pol1}-\eqref{3new-pol5} satisfies the following asymptotic behavior estimations
\begin{equation}\label{1COR}
\int_0^{L}\abs{u^1_x}^2dx=o(1)\quad \text{and}\quad \int_0^L\abs{u^3_x}^2dx=\frac{o(1)}{\la^{\frac{1}{2}}}.
\end{equation}
and the following estimations
\begin{equation}\label{2COR}
\int_0^L\abs{\la u^1}^2dx=o(1),\quad \int_0^L\abs{\la u^3}^2dx=\frac{o(1)}{\la^{\frac{1}{2}}}\quad \text{and}\quad \int_0^L\abs{\la y^3}^2dx=o(1).
\end{equation}
\end{Corollary}
\begin{proof}
Using \eqref{C3-Lemma5-1}, \eqref{COR1} we obtain the first estimation in \eqref{1COR}. Similarly, using \eqref{C3-Lemma5-1} and \eqref{9C3-Lemma3}, we get the second estimation in \eqref{1COR}. Now, multiplying \eqref{3new-pol1}, \eqref{3new-pol2} and \eqref{3new-pol5}  by $\overline{u^1}$, $\overline{u^3}$ and $\overline{y^3}$ respectively, integrating by parts and using \eqref{COR1} and \eqref{1COR}, we get \eqref{2COR}. The proof has been completed. 
\end{proof}
$\newline$\\[0.1in]
{\noindent We now finish the proof of Theorem \ref{POL-bc}}. Using Lemmas \ref{3C3-diss}, \ref{C3-L2-est}  and equations \eqref{COR1}, \eqref{1COR} and \eqref{2COR}, we obtain that {$\|U\|_{\mathcal{H}}=o(1)$ }, which contradicts $(H_2)$. This implies that 
$$
\sup_{\la\in \R}\|(i\la I-\mathcal{A})^{-1}\|_{\mathcal{H}}{\le}O(\la^6).
$$
The proof has been completed.\\[0.1in]  

{\noindent To conclude this paper, we give the following observations: \\
 For a system of partially damped wave equations coupled by velocity or displacement (such as the Timoshenko beam equation, the Bresse beam equation), it is well-known that the damping is more effective in the case of equal wave speeds. However, for the Rao-Nakra sandwich beam equation or its generalized version, the opposite is true. The key to stabilization is to break up the symmetry.
\begin{itemize}
 	\item When both the longitudinal and shear displacements of the top/bottom layer are damped, no condition on the system coefficients is needed for stabilizability. 
 	\item When only the longitudinal displacement of the top/bottom layer is damped, we need the shear modulus $G_1\ne G_3$ for stabilizability.
 	\item When only the shear displacement of the top/bottom layer is damped, we need $\frac{E_1}{\rho_1}\ne \frac{E_3}{\rho_3}$ for stabilizability.   
\end{itemize} }

\appendix
\section{Some notions and stability theorems}\label{p2-appendix}
\noindent In order to make this paper more self-contained, we recall in this short appendix some notions and stability results used in this work. 
\begin{defi}\label{App-Definition-A.1}{\rm
		Assume that $A$ is the generator of $C_0-$semigroup of contractions $\left(e^{tA}\right)_{t\geq0}$ on a Hilbert space $H$. The $C_0-$semigroup $\left(e^{tA}\right)_{t\geq0}$ is said to be 
		\begin{enumerate}
			\item[$(1)$] Strongly stable if 
			$$
			\lim_{t\to +\infty} \|e^{tA}x_0\|_H=0,\quad \forall\, x_0\in H.
			$$
			\item[$(2)$] Exponentially (or uniformly) stable if there exists two positive constants $M$ and $\varepsilon$ such that 
			$$
			\|e^{tA}x_0\|_{H}\leq Me^{-\varepsilon t}\|x_0\|_{H},\quad \forall\, t>0,\ \forall\, x_0\in H.
			$$
			\item[$(3)$] Polynomially stable if there exists two positive constants $C$ and $\alpha$ such that 
			$$
			\|e^{tA}x_0\|_{H}\leq Ct^{-\alpha}\|A x_0\|_{H},\quad \forall\, t>0,\ \forall\, x_0\in D(A).
			$$
			\xqed{$\square$}
	\end{enumerate}}
\end{defi}
\noindent To show  the strong stability of the $C_0$-semigroup $\left(e^{tA}\right)_{t\geq0}$ we rely on the following result due to Arendt-Batty \cite{Arendt01}. 
\begin{Theorem}\label{App-Theorem-A.2}{\rm
		{Assume that $A$ is the generator of a C$_0-$semigroup of contractions $\left(e^{tA}\right)_{t\geq0}$  on a Hilbert space $H$. If $A$ has no pure imaginary eigenvalues and  $\sigma\left(A\right)\cap i\mathbb{R}$ is countable,
			where $\sigma\left(A\right)$ denotes the spectrum of $A$, then the $C_0$-semigroup $\left(e^{tA}\right)_{t\geq0}$  is strongly stable.}\xqed{$\square$}}
\end{Theorem}
\noindent Concerning the characterization of polynomial stability stability of a $C_0-$semigroup of contraction $\left(e^{tA}\right)_{t\geq 0}$ we rely on the following result due to Borichev and Tomilov \cite{Borichev01} (see also \cite{Batty01} and \cite{RaoLiu01})
\begin{Theorem}\label{bt}
	{\rm
		Assume that $A$ is the generator of a strongly continuous semigroup of contractions $\left(e^{tA}\right)_{t\geq0}$  on $\mathcal{H}$.   If   $ i\mathbb{R}\subset \rho(\mathcal{A})$, then for a fixed $\ell>0$ the following conditions are equivalent
		\begin{equation}\label{h1}
		\sup_{\lambda\in\mathbb{R}}\left\|\left(i\lambda I-\mathcal{A}\right)^{-1}\right\|_{\mathcal{L}\left(\mathcal{H}\right)}=O\left(|\lambda|^\ell\right),
		\end{equation}
		\begin{equation}\label{h2}
		\|e^{t\mathcal{A}}U_{0}\|^2_{\HH} \leq \frac{C}{t^{\frac{2}{\ell}}}\|U_0\|^2_{D(\AA)},\hspace{0.1cm}\forall t>0,\hspace{0.1cm} U_0\in D(\AA),\hspace{0.1cm} \text{for some}\hspace{0.1cm} C>0.
		\end{equation}\xqed{$\square$}}
\end{Theorem}

\end{document}